%% file: lsnou_190715.tex
\documentclass[final,1p,times]{elsarticle}

\usepackage{amssymb}
\usepackage{amsmath}
 \usepackage{amsthm}
  \usepackage{color}

\usepackage{subfigure}

\newtheorem{thm}{Theorem}[section]

\newtheorem{lem}[thm]{Lemma}

\newtheorem{rem}[thm]{Remark}

\newtheorem{asm}[thm]{Assumption}

\def\ba{{\mathbf{a}}}

\def\bE{{\bf E}}

\def\bW{{\bf W}}

\newcommand{\C}{\rm I\kern-.5emC}
\newcommand{\R}{\rm I\kern-.19emR}

\def\3bar{{|\hspace{-.02in}|\hspace{-.02in}|}}

\journal{...}

\begin{document}

\begin{frontmatter}

\title{Adaptive Flux-Only Least-Squares Finite Element Methods for Linear Transport Equations}

\author{Qunjie Liu} 
\ead{qjliu2-c@my.cityu.edu.hk}
\address{Department of Mathematics, 
City University of Hong Kong, Hong Kong SAR, China}

\author{Shun Zhang \corref{cor1}}
\ead{shun.zhang@cityu.edu.hk}
\address{Department of Mathematics, 
City University of Hong Kong, Hong Kong SAR, China}

\cortext[cor1]{Corresponding author}


\begin{abstract}
In this paper, two flux-only least-squares finite element methods (LSFEM) for the linear hyperbolic transport problem are developed. The transport equation often has discontinuous solutions and discontinuous inflow boundary conditions, but the normal component of the flux across the mesh interfaces is continuous. In traditional LSFEMs, the continuous finite element space  is used to approximate the solution. This will cause unnecessary error around the discontinuity and serious overshooting. In \cite{LZ:18}, we reformulate the equation by introducing a new flux variable to separate the continuity requirements of the flux and the solution. Realizing that the Raviart-Thomas mixed element space has enough degrees of freedom to approximate both the flux and its divergence, we eliminate the solution from the system and get two flux-only formulations, and develop corresponding LSFEMs. The solution then is recovered by simple post-processing methods using its relation with the flux. These two versions of flux-only LSFEMs use less DOFs than the method we developed in \cite{LZ:18}.
	
Similar to the LSFEM developed in \cite{LZ:18}, both flux-only LSFEMs can handle discontinuous solutions better than the traditional continuous polynomial approximations. We show the existence, uniqueness, a priori and a posteriori error estimates of the proposed methods.  With adaptive mesh refinements driven by the least-squares a posteriori error estimators, the solution can be accurately approximated even when the mesh is not aligned with discontinuity. The overshooting phenomenon is very mild if a piecewise constant reconstruction of the solution is used. Extensive numerical tests are done to show the effectiveness of the methods developed in the paper.
\end{abstract}
\end{frontmatter}

\input{defs.tex}
\section{Introduction}\label{intro}
\setcounter{equation}{0}
In this paper, we consider the following linear transport equation in the conservative form:
\begin{align}
 \label{transporteqn}
\gradt(\bbeta u)+\gamma u &= f \quad \mbox{in} \,\ \O, \\ \nonumber
                        u &= g \quad \mbox{on} \,\ \Gamma_{-},
\end{align}
with $\bbeta$ an advection field and $\Gamma_-$ the inflow boundary. Detailed descriptions of the equation can be found in Section 2.

It is crucial to realize that unlike the elliptic or parabolic equations whose solutions are generally smooth, hyperbolic equations commonly have discontinuous solutions. For the transport equation, the solution space is 
$$
W = \{ v \in L^2(\O): \gradt (\bbeta v) \in L^2(\O)\}.
$$
For the space $W$, there is no simple finite element subspace of it. The usual continuous finite element space requires too much continuity while the discontinuous finite element space lacks the continuity of the flux $\bbeta u$. To separate the continuity requirement,  we realize that for a true solution $u\in W$, the condition $\gradt(\bbeta u) \in L^2(\O)$ essentially means that
$$
u \in L^2(\O) \quad \mbox{and}\quad\bbeta u \in H(\divvr;\O).
$$ 
In \cite{LZ:18}, we propose new least-squares variational formulations with a flux reformulation. Introducing the flux $\bsigma = \bbeta u \in H(\divvr;\O)$, we have a first order system (with appropriate boundary conditions):
\beq \label{1stordersys}
\bsigma -\bbeta u =0 \quad \mbox{and}\quad \gradt \bsigma +\gamma u =f. 
\eeq
The space requirement of $u$ is reduced to $L^2(\O)$ only. In \cite{LZ:18}, we use $ H(\divvr;\O)$-conforming Raviart-Thomas finite element space $RT_k$ and discontinuous piecewise polynomial space $P_k$ to approximate $\bsigma\in H(\divvr;\O)$ and $u\in L^2(\O)$ separately. Notice that the space $RT_k$ has enough degrees of freedom to approximate both $\bsigma$ and the divergence of $\bsigma$ in the same order $k$, this suggests us that it is possible to remove the $u$-variable from the system and get a system with the flux only to reduce the degrees of freedom of the discrete system.

In this paper, two reformulations are suggested. In the first reformulation, under the assumption that $|\bbeta| \neq 0$, we use $\bsigma = \bbeta u$ to get $u=\bsigma \cdot \bbeta/{|\bbeta|^2}$, then a single equation for $\bsigma$ is obtained:
$$
\gradt \bsigma +\gamma   \dfrac{\bbeta \cdot \bsigma}{|\bbeta|^2}=f. 
$$
An extra orthogonal condition is also added to ensure the uniqueness of the flux $\bsigma$. 

The second reformulation uses $\gradt \bsigma + \gamma u =f$ to eliminate $u$ with the assumption $\gamma \neq 0$. Then we can have a single equation for $\bsigma$:
$$
\gamma \bsigma + \bbeta (\gradt\bsigma -f) = 0. 
$$
Based on these two reformulations with the flux variable only, two least-squares finite element methods are developed in the paper. A simple post-processing using either  $u=\bsigma \cdot \bbeta/{|\bbeta|^2}$ or $u = (f-\gradt \bsigma)/\gamma$ can get an approximation of the original variable $u$.

LSFEMs for the transport equations can be found in \cite{CJ:88,Jiang:98,BochevChoi:01,BC:01,DMMO:04,BG:09,BG:16}. In these papers, the $u$-only formulation is used and continuous finite element space is used to approximate it. Since continuous functions are used to approximate discontinuous solutions, it will introduce unnecessary error. The error indicator will always indicate unnecessary big errors for those elements on the discontinuity if an adaptive method is used. Compared with a $P_0$ discontinuous approximations for $u$ which is easily available for the methods in this paper, Gibbs phenomenon like spurious over-shooting is more intense near the discontinuity in these continuous approximations, see also discussions in \cite{Zhang:19} on overshooting with adaptive finite elements.

The nonconforming LSFEM in \cite{DHSW:12} and the similar method in \cite{MY:18} use discontinuous approximations with a penalty term to enforce the continuity of the normal component of the flux. The $H(\divvr)$ flux-only methods use natural finite element spaces without any penalty terms. In addition, these two methods have to use at least piecewise linear functions, which can lead to non-trivial overshooting, see discussions in \cite{Zhang:19}. The methods developed in this paper can simply recover the solution in a piecewise constant space to avoid overshooting.

The paper is organized as follows. Section 2 describes the model linear hyperbolic transport problem. Based on two flux-only reformulations, two least-squares variational problem are presented in Section 3. Corresponding LSFEMs are developed in Section 4, a priori and a posteriori error estimates are established. Section 5 provides numerical results for many test problems. In Section 6, we make some concluding remarks.

\section{Model Linear Hyperbolic Transport Equation}
\setcounter{equation}{0}

Let $\Omega$ be a bounded polyhedral domain in $\Re^d$ ($d= 2$ or $3$) with a Lipschitz boundary. Assume that 
\begin{align}
\gamma(x) \in L^{\infty}(\O) \quad  \mbox{and}\quad \bbeta(x)=(\beta_1,\cdots, \beta_d)^{T} \in [C^1(\overline{\O})]^d.
\end{align}
Define the inflow and outflow parts of $\p \O$ in the usual fashion:
\begin{align*}
&\Gamma_{-}=\{x\in\p \O:\bbeta(x)\cdot\bn_{\p\O}(x)<0\}=\mbox{inflow}, \\
&\Gamma_{+}=\{x\in\p\O:\bbeta(x)\cdot\bn_{\p\O}(x)>0\}=\mbox{outflow},
\end{align*}
where $\bn_{\p\O}(x)$ denotes the unit outward normal vector to $\p\O$ at $x\in\p\O$. In order to specify the boundary condition in (\ref{transporteqn}), define
$$
L^2(|\bbeta\cdot\bn|;\Gamma_-):= \{
v \mbox{ is measurable on } \p\O : \int_{\Gamma_-}|\bbeta\cdot\bn| v^2 < \infty\}.
$$

In this paper, we assume that the coefficients are nice enough to guarantee the existence and uniqueness of the solution.

\begin{asm} \label{asmbeta}
{\bf  (Assumption on the data $\bbeta$ and $\gamma$)}
Assume that for $g \in L^2(|\bbeta\cdot\bn|;\Gamma_-)$,
the linear transport equation (\ref{transporteqn}) has a unique solution
in $W$ for the given data $\bbeta$ and $\gamma$.
\end{asm}
Several known conditions that guarantee the existence and uniqueness of the solution are:
\begin{enumerate}[(i)]
\item There exists a positive $\gamma_0$, such that
$$
\gamma + \dfrac{1}{2}\gradt \bbeta \geq \gamma_0 > 0 \quad \mbox{in  } \Omega.
$$
\item   $\bbeta$ is a nonzero constant vector with $\gamma=0$. 
\end{enumerate}
The proof of the uniqueness and existence with the condition (i)  can be founded in \cite{DHSW:12} and Chapter 2 of \cite{DE:12}. For the case with the condition (ii) (and more general cases), the proof is based on the standard ODE theory, and can be founded in \cite{DHSW:12,DMMO:04}.

%
%
%

\begin{rem}
An equivalent non-conservative reformulation is
\begin{align}
\bbeta\cdot \nabla u+\mu u &= f \quad \mbox{in} \,\ \O, \\ \nonumber
                        u &= g \quad \mbox{on} \,\ \Gamma_{-},
\end{align}
with $\mu = \gamma + \gradt \bbeta$. We can apply the methods developed in this paper to the non-conservative version by  changing it to the conservative formulation.
\end{rem}

\section{Flux-Only Least-Squares Variational Problems}
\setcounter{equation}{0}

In this section, two flux only least-squares variational problems are introduced. The existence and uniqueness of the formulations are discussed.

\subsection{Reformulation with both flux $\bsigma$ and solution $u$}

In \cite{LZ:18}, a flux $\bsigma=\bbeta u$ is introduced to separate the continuity requirements, then the divergence of the flux is in $L^2(\O)$, $\gradt \bsigma =f - \gamma u \in L^2(\O)$. Thus the flux $\bsigma \in H(\divvr;\O)$. Since the inflow condition is essentially a condition of the flux, we write the condition $u = g$ on $\Gamma_{-}$ as
$$
\bsigma \cdot \bn = (\bbeta \cdot \bn) g, \quad \mbox{on} \,\ \Gamma_{-}.
$$
Thus we get the following system of $\bsigma$ and $u$:
\beq \label{1stordersys2}
\left\{
\begin{array}{rlll}
\bsigma -\bbeta u &=& 0 
		& \mbox{in  } \O, \\[2mm]
\gradt \bsigma + \gamma u &=& f     & \mbox{in  }  \O, \\[2mm]
\bsigma \cdot \bn  &=&  (\bbeta \cdot \bn) g  & \mbox{on  }  \Gamma_{-}.
\end{array}
\right.
\eeq
Define the following spaces:
\begin{eqnarray*}
H_{g,-}(\divvr;\O) &:=& \{ \btau \in H(\divvr;\O) : \btau\cdot\bn = (\bbeta \cdot \bn) g \mbox{  on  }\Gamma_-\}, \\
H_{0,-}(\divvr;\O) &:=&  \{ \btau \in H(\divvr;\O) : \btau\cdot\bn = 0 \mbox{  on  }\Gamma_-\}.
\end{eqnarray*}

\subsection{Flux-only least-squares variational problem: first reformulation}
In this reformulation, we assume:
\beq
	|\bbeta(x)|>0, \quad \forall \, x\in \O.
\eeq
From the equation $\bsigma =\bbeta u$, we have
$$
	u=\dfrac{\bbeta \cdot \bsigma}{|\bbeta|^2}.
$$
Then we obtain a single equation of $\bsigma$: 
$$
\gradt \bsigma +\gamma   \dfrac{\bbeta \cdot \bsigma}{|\bbeta|^2}=f. 
$$
Such $\bsigma$ is not unique, since a single scalar equation is used to compute a vector function $\bsigma$. For example, when $\gamma=0$, the only requirement is $\gradt \bsigma =f$. Thus some extra conditions are needed to ensure the uniqueness of $\bsigma$. Notice that since $\bsigma = \bbeta u$, we should also have
$$
 \bsigma \cdot \bbeta_{\bot}^{(i)}=  \bbeta u \cdot \bbeta_{\bot}^{(i)} =0,\ \forall i\in I: =\{1,\cdots,d-1\},
$$
where $\{\bbeta_{\bot}^{(i)}\}_{i=1}^{d-1}$ is an orthonormal family of vectors satisfying
$$
\bbeta(x) \cdot \bbeta_{\bot}^{(i)}(x)=0, \ \forall i \in I,
\quad \mbox{and} \quad  
\bbeta_{\bot}^{(i)}(x)\cdot\bbeta_{\bot}^{(j)}(x)=\delta_{ij}, \ \forall i,j\in I. 
$$
For $d=2$, we can simply choose $\bbeta_{\bot}(x)=(-\beta_2(x),\beta_1(x))/|\bbeta(x)|$.
For $d=3$, $\bbeta_{\bot}^{(1)}(x)$ and $\bbeta_{\bot}^{(2)}(x)$ can be obtained 
by first finding two linearly independent vector functions orthogonal to $\bbeta(x)$,
then performing a Gramm-Schmidt orthonormalization procedure.

Define
\beq
\tilde{\gamma} =   \gamma/|\bbeta|^2.
\eeq
We get the following linear system with respect to $\bsigma$:
\beq\label{transportequation}
\left\{
\begin{array}{rlll}
\gradt \bsigma + \tilde{\gamma} (\bbeta \cdot \bsigma)  &=& f 
		& \mbox{in  } \O, \\[2mm]
                                   \bsigma \cdot \bbeta_{\bot}^{(i)} &=&0     & \mbox{in  }  \O, \ i\in I, \\[2mm]
                                   \bsigma \cdot \bn  &=&  (\bbeta \cdot \bn) g  & \mbox{on  }  \Gamma_{-}.
\end{array}
\right.
\eeq
The least-squares minimization problem of \eqref{transportequation} is:
we seek $\bsigma \in H_{g,-}(\divvr;\O)$, such that
\beq \label{LS}
\cJ_1(\bsigma;f,g) = \inf_{\btau \in H_{g,-}(\divvr;\O) } \cJ_1(\btau;f,g),
\eeq
with the least-squares functional $\cJ_1$ defined as
\beq
\cJ_1(\btau;f,g) := 
\left \|\gradt \btau+ \tilde{\gamma} (\bbeta \cdot \btau)- f \right\|_0^2 
+ \sum_{i=1}^{d-1}\|\btau \cdot \bbeta_{\bot}^{(i)} \|_0^2, \quad \forall \, \btau \in H_{g,-}(\divvr;\O).
\eeq
Its corresponding Euler-Lagrange formulation is: find $\bsigma \in H_{g,-}(\divvr;\O)$, such that
\beq \label{LS_EL}
a_1(\bsigma,\btau) = f_1(\btau),	\quad \forall \, \btau \in H_{0,-}(\divvr;\O),
\eeq
where the bilinear form $a_1$ is defined as
$$
a_1(\btau,\brho) := \left(\gradt \btau + \tilde{\gamma} (\bbeta \cdot \btau), 
	\gradt \brho + \tilde{\gamma}\bbeta \cdot \brho \right) 
	+ \sum_{i=1}^{d-1}\left(\btau \cdot \bbeta_{\bot}^{(i)}  ,\brho \cdot \bbeta_{\bot}^{(i)}  \right), 
$$	
for all $\btau, \brho \in H(\divvr;\O)$, and the linear form $f_1$ is defined as
$$
	f_1(\btau) = (f, \gradt \btau + \tilde{\gamma} (\bbeta \cdot \btau) ), \quad \forall\ \btau  \in H(\divvr;\O).
$$

\begin{lem} \label{norm1}
Assuming that the coefficients $\bbeta$ and $\gamma$ satisfy Assumption \ref{asmbeta}, the following defines a norm for $\btau \in H_{0,-}(\divvr;\O)$:
\beq
\tri \btau \tri_1 :=  \left(\|\gradt \btau +\tilde{\gamma} (\bbeta \cdot \btau)\|_0^2 + \sum_{i=1}^{d-1}\| \btau \cdot \bbeta_{\bot}^{(i)}\|_0^2\right )^{1/2}. 
\eeq
\end{lem}
\begin{proof}
The linearity and the triangle inequality are obvious for the definition $\tri \btau\tri_1$.

When $\tri \btau\tri_1 = 0$, we have
$$
\gradt \btau +\tilde{\gamma} (\bbeta \cdot \btau) = 0 
\quad\mbox{and}\quad \btau \cdot \bbeta_{\bot}^{(i)}=0, \ \forall i\in I.
$$
Note that $\bbeta \cdot \bbeta_{\bot}^{(i)}=0, \ \forall i\in I$, then $\bbeta(x)$ and $\btau(x)$ are linear dependent for all $x\in \O$. Thus there must exist a function $v(x)$ such that $\btau(x)=\bbeta (x)v(x)$ for all $x\in \O$. Substituting this into the first equation yields $\gradt(\bbeta v)+\gamma v=0$. 
On the inflow boundary, $\btau\cdot\bn = 0$ on $\Gamma_-$, thus we get $\bbeta\cdot\bn v = 0$. Since $\bbeta\cdot\bn \neq 0$ on $\Gamma_-$, $v$ is identically $0$ on $\Gamma_-$. 
By Assumption \ref{asmbeta}, the equation has a unique solution, and the solution has to be  $v=0$. So we have $\btau = \bbeta v = 0$. Thus the norm $\tri \cdot\tri_1$ is well defined.
\end{proof}

\begin{rem}
It is also clear that
$$
\tri \btau\tri_{1,K} := 
\left(\|\gradt \btau +\tilde{\gamma} (\bbeta \cdot \btau)\|_{0,K}^2 
+ \sum_{i=1}^{d-1}\| \btau \cdot \bbeta_{\bot}^{(i)}\|_{0,K}^2
\right )^{1/2}
$$
is a semi-norm on an element $K\in\cT$.
\end{rem}

%
\begin{thm} \label{EU1}
The least-squares problem (\ref{LS}) has a unique solution $\bsigma \in H_{g,-}(\divvr;\O)$ with the assumption $g\in L^2(|\bbeta\cdot\bn|;\Gamma_-)$ and the data $\bbeta$ and $\gamma$ satisfying  Assumption \ref{asmbeta}.
\end{thm}
\begin{proof}
For the existence, with the assumption of $g\in L^2(|\bbeta\cdot\bn|;\Gamma_-)$ and Assumption \ref{asmbeta}, there exists a $\tilde{u}\in W\subset L^2({\O})$, such that $\tilde{u} = g$ on $\Gamma_-$ satisfying  (\ref{transporteqn}). Let
$\tilde{\bsigma} = \bbeta \tilde{u}$, then
$$
\tilde{u} = \dfrac{\bbeta \cdot \tilde{\bsigma}}{|\bbeta|^2}, \quad
\|\tilde{\bsigma}\|_0 \leq \|\bbeta\|_{\infty} \|\tilde{u}\|_0, \quad
\|\gradt \tilde{\bsigma} \|_0 = \|f-\gamma \tilde{u}\|_0 \leq \|f\|_0 + \|\gamma\|_{\infty}\|\tilde{u}\|_{0}.
$$
Also, on the inflow boundary, $\tilde{\bsigma}\cdot\bn = \bbeta\cdot\bn \tilde{u}= (\bbeta\cdot\bn)g$. Thus $\tilde{\bsigma} \in H_{g,-}(\divvr;\O)$ and satisfies (\ref{transportequation}). So $\tilde{\bsigma}\in H_{g,-}(\divvr;\O)$ is the minimizer such that $\cJ_1(\tilde{\bsigma};f,g) =0$.

To prove the uniqueness, let $\bsigma_1$ and $\bsigma_2$ in $H_{g,-}(\divvr;\O)$ be two solutions of (\ref{LS}) or (\ref{LS_EL}). Let $E = \bsigma_1 -\bsigma_2 \in H_{0,-}(\divvr;\O)$, then $\tri E \tri_1 ^2 = a_1(E,E) =0$. Since in Lemma \ref{norm1}, we already showed that $\tri \cdot \tri_1$ is a norm, the uniqueness is proved.
\end{proof}

\subsection{Flux-only least-squares variational problem: second reformulation}
If we assume that $\gamma\neq 0$, or more precisely, we assume that 
$$
\gamma (x) \neq 0, \quad \forall x\in \O \quad a.e..
$$ 
We can use $u = \dfrac{1}{\gamma} (f-\gradt\bsigma)$ to eliminate $u$ and  get
\beq
\bsigma + \dfrac{\bbeta}{\gamma} (\gradt\bsigma -f) = 0,
\eeq
by $\bsigma = \bbeta u$.
Notice this is a vector equation, thus this reformulation does not need the extra orthogonal condition to guarantee the uniqueness of $\bsigma$. We can also reformulate it as 
\beq
\gamma \bsigma + \bbeta (\gradt\bsigma -f) = 0. 
\eeq
We get the following linear system with respect to $\bsigma$ for $\gamma\neq 0$:
\beq\label{transportequation2}
\left\{
\begin{array}{rlll}
\gamma \bsigma + \bbeta (\gradt\bsigma -f)   &=& 0
		& \mbox{in  } \O, \\[2mm]
 \bsigma \cdot \bn  &=&  (\bbeta \cdot \bn) g  & \mbox{on  }  \Gamma_{-}.
\end{array}
\right.
\eeq
The least-squares variational problem of \eqref{transportequation2} is:
we seek $\bsigma \in H_{g,-}(\divvr;\O)$, such that
\beq \label{LS2}
\cJ_2(\bsigma;f,g) = \inf_{\btau \in H_{g,-}(\divvr;\O) } \cJ_2(\btau;f,g),
\eeq
with the least-squares functional $\cJ_2$ defined as
\beq
\cJ_2(\btau;f,g) := 
\left \| \gamma \btau + \bbeta (\gradt\btau -f)  \right\|_0^2, \quad \forall\ \btau \in H_{g,-}(\divvr;\O).
\eeq
Its corresponding Euler-Lagrange formulation is: find $\bsigma \in H_{g,-}(\divvr;\O)$, such that
\beq \label{LS_EL}
a_2(\bsigma,\btau) = f_2(\btau), 	\quad \forall \, \btau \in H_{0,-}(\divvr;\O),
\eeq
where the bilinear form $a_2$ is defined as
\begin{eqnarray*}
a_2(\btau,\brho) &:= &\left(\gamma \btau + \bbeta \gradt\btau, 
	\gamma \brho + \bbeta \gradt\brho \right) \\
	&=&(\gamma^2 \btau, \brho) +(|\bbeta|^2 \gradt\btau,\gradt \brho) +(\gamma \bbeta\cdot\btau, \gradt \brho)
	+(\gamma \bbeta\cdot\brho, \gradt \btau),
\end{eqnarray*}
for all $\btau, \brho \in H(\divvr;\O)$,
and
$$
f_2(\btau) := 
 (\bbeta f,  \bbeta \gradt \btau +\gamma \btau ) = (f,  |\bbeta|^2 \gradt \btau +\gamma \bbeta\cdot \btau ).
$$
\begin{rem}
Note that $f_1(\btau)$ and $f_2(\btau)$ are essentially the same if we weight the system with $|\bbeta|^2$ (or simply normalize $\bbeta$ to make $|\bbeta| =1$), but the $a_1$ and $a_2$ are not the same when $\gamma \neq 0$ since 
$$
(\bbeta \cdot \btau, \bbeta \cdot \brho) \neq (|\bbeta|^2 \btau, \brho). 
$$
Thus these two reformulations are indeed different.
\end{rem}

\begin{lem} \label{norm2}
Assuming that the coefficients $\bbeta$ and $\gamma \neq 0$ satisfy the assumption of the data ensuring the existence and uniqueness of original equation (\ref{transporteqn}), the following defines a norm for $\btau \in H_{0,-}(\divvr;\O)$:
\beq
\tri \btau \tri_2 := 
\|\gamma \btau + \bbeta \gradt \btau\|_0. 
\eeq
\end{lem}
\begin{proof}
We only need to check that $\btau =0$ when $\tri \btau\tri_2 = 0$.

If $\tri \btau\tri_2 = 0$, we have
$$
\gamma \btau + \bbeta \gradt \btau = 0.
$$
Choose $v = -\frac{1}{\gamma} \gradt\btau \in L^2(\O)$, then we have $\gamma \btau - \bbeta \gamma v =0$. Since $\gamma \neq 0$, $\btau = \bbeta v$. Thus we have a pair $(\btau,v)$ satisfying
$$
\gradt\btau + \gamma v =0 \quad\mbox{and}\quad \btau = \bbeta v.
$$
This yields $\gradt(\bbeta v)+\gamma v=0$. Similar to Lemma \ref{norm1},  we have $v = 0$ on $\Gamma_-$, and thus  $v=0$ in $\O$ and $\btau = \bbeta v = 0$. The lemma is proved.
\end{proof}

\begin{thm} \label{EU2}
The least-squares problem (\ref{LS2}) has a unique solution $\bsigma \in H_{g,-}(\divvr;\O)$ with
the assumption $g\in L^2(|\bbeta\cdot\bn|;\Gamma_-)$ 
and the coefficients $\bbeta$ and $\gamma \neq 0$ satisfying Assumption \ref{asmbeta}.
\end{thm}
\begin{proof}
Similar to the proof of Theorem \ref{EU1},
there exists a $\tilde{u}\in W\subset L^2({\O})$, such that $\tilde{u} = g$ on $\Gamma_-$ satisfying  (\ref{transporteqn}). 
Let $\tilde{\bsigma} = \bbeta \tilde{u}$, then $\gradt\tilde{\bsigma} +\gamma \tilde{u} =f$,
thus $\gamma \tilde{\bsigma} = \bbeta (f-\gradt\tilde{\bsigma})$.
Similar to the proof of Theorem \ref{EU1}, $\tilde{\bsigma} \in H_{g,-}(\divvr;\O)$.
So 
$\tilde{\bsigma}$ is the minimizer such that
$
\cJ_2(\tilde{\bsigma};f,g) =0.
$
The uniqueness is  due to that $\tri \cdot \tri_2$ is a norm.
\end{proof}

\begin{rem}
In traditional least-squares methods, a norm equivalence is often seeked to prove to the existence and uniqueness. As we will see, it is not possible in our case. Thus an indirect proof is used. Another example of such least-squares method is its application in non-divergence equation, see \cite{QZ:19}.
\end{rem}

\section{Flux-only Least-Squares Finite Element Methods}
\setcounter{equation}{0}
In this section, least-squares finite element methods based on the flux-only least-squares variational problems are developed. A priori and a posteriori error estimates are derived.

Let $\cT = \{K\}$ be a triangulation of $\O$ using simplicial elements. The mesh $\cT$ is assumed to be regular. We denote the set of edges/faces of the triangulation $\cT$ on inflow boundary $\Gamma_-$ by $\cE_{-}$. For an element $K\in \cT$ and an integer $k\geq 0$, let $P_k(K)$ be the space of polynomials with degrees less than or equal to $k$. The space $P_k(F)$ is defined similarly on an edge/face $F$. Define the finite element spaces $RT_k$ and $P_k$ as follows:
$$
RT_k  :=\{\btau \in H(\divvr; \O) \colon
			\btau|_K \in P_k(K)^d +\bx P_k(K),\,\, \forall \, K\in \cT\},
$$
and
$$			
P_k :=	\{v \in L^2(\O) \colon
			v|_K \in P_k(K),\,\, \forall \, K\in \cT\}.	
$$

\begin{asm}
({\bf Assumption on the boundary data})
For simplicity, we assume $(\bbeta\cdot\bn)g$ on $\Gamma_-$ can be approximated exactly by the trace of $RT_k$ space on $\Gamma_-$, i.e., $g|_F\in P_k(F)$, for all faces/edges $F\in \cE_{-}$. 
\end{asm}
Note that this assumption still allows a discontinuous boundary condition, but it does require that the boundary mesh is aligned with the discontinuity. For an arbitrary $g$, we need to first interpolate or project $(\bbeta\cdot\bn)g$ to the piecewise polynomial space.

Define
$$
RT_{k,g,-}:=\{\btau \in RT_k \colon
			\btau\cdot\bn = (\bbeta\cdot\bn)g \mbox{  on  } \Gamma_-\}.
$$

\subsection{Interpolations and their properties}
In order to derive  a priori error estimates, we introduce some interpolations and their properties. Note that all properties here are local with respect to the element $K$.

For  $s>0$, denote by $I^{rt}_k: \Hdiv \cap [H^s(\O)]^d \mapsto RT_k$
the standard $RT$ interpolation operator \cite{BBF:13}. It satisfies
the following approximation property: for $\btau \in H^{s}(K)^d$, $s>0$,
\beq \label{rti}
  \|\btau - I^{rt}_k \btau\|_{0,K}
  \leq C h_K^{\min\{s,k+1\}} |\btau|_{{\min\{s,k+1\}},K}, \quad\forall\,\, K\in \cT.
\eeq
 In addition, another approximation property holds: for $\btau \in H^{s}(K)$
and $\gradt\btau \in H^s(K)$, $s>0$,
\beq \label{rti2}
  \|\gradt(\btau - I^{rt}_k \btau)\|_{0,K}
  \leq C h_K^{\min\{s,k+1\}} |\gradt\btau|_{{\min\{s,k+1\}},K}, \quad\forall\,\, K\in \cT.
\eeq
The discussion of the above two properties can be found in \cite{LZ:18}.


\subsection{Least-squares finite element methods}

\noindent ({\bf Flux-only LSFEM Problems})
We seek $\bsigma_h \in RT_{k,g,-}$,
such that
\beq \label{LSFEM}
\cJ_i(\bsigma_h;f,g) = \inf_{\btau \in RT_{k,g,-} } \cJ_i(\btau;f,g),  \quad i= 1 \mbox{ or } 2.
\eeq
Or equivalently, find $\bsigma_h \in RT_{k,g,-}$, such that
\beq \label{LSFEM_EL}
a_i(\bsigma_h;\btau) = f_i(\btau), \quad \forall\, \btau \in RT_{k,0,-}, \quad i= 1 \mbox{ or } 2.
\eeq
We call LSFEM1 and LSFEM2 for least-squares finite element methods with $i = 1$ and $2$ separately. Note that we need the assumption $|\bbeta|>0$ for LSFEM1, and the assumption $\gamma \neq 0$ is needed for LSFEM2.

\subsection{Post-processing for $u_h$}
Once the numerical flux $\bsigma_h$ is known, since we have two equations to relate $u$ and $\bsigma$, there are two simple ways to recover a discrete approximation $u_h$:
\beq \label{pp_u1}
	\mbox{The first recovery} \quad 
u_h= \frac{\bsigma_h\cdot\bbeta}{|\bbeta|^2}, \quad \mbox{assuming  } |\bbeta| >0.
\eeq
\beq \label{pp_u2}
	\mbox{The second recovery} \quad 
	u_h= \frac{1}{\gamma} (f-\gradt\bsigma_h), \quad \mbox{assuming  }  \gamma \neq 0.
\eeq
Note that both post-processing procedures can be used for both LSFEMs as long as the assumption of the coefficients holds. 

We can also use a simple element-wise $L^2$-projection to get an approximation $u_h \in P_k$ if needed. To ensure the overshoot effect is mild, we should project $u_h$ onto $P_0$ as suggested in \cite{Zhang:19}.

\subsection{A priori error estimates}
The following best approximation property is simple and straight.
\begin{thm} \label{thm_bestapp}
(Cea's lemma type of result)
Let $\bsigma$ be the solution of least-squares variational problems \eqref{LS} or \eqref{LS2}, 
and $\bsigma_h$ be the solution of the corresponding LSFEM \eqref{LSFEM} satisfying the corresponding condition on $\bbeta$ and $\gamma$ and the assumption on the boundary data, the following best approximation result holds:
\beq
\tri \bsigma-\bsigma_h\tri_i \leq \inf _{\btau_h \in RT_{k,g,-} } \tri \bsigma-\btau_h\tri_i, \quad i=1 \mbox{ or }2.
\eeq
\end{thm}
\begin{proof}
For an arbitrary $\btau_h \in RT_{k,0,-} $, the following error equation holds:
$$
a_i(\bsigma-\bsigma_h, \btau_h) =0, \quad \forall \btau_h \in RT_{k,0,-}.
$$
From the definition of the norm $\tri \cdot \tri_i$, the error equation, and the Cauchy-Schwarz inequality, we have
\begin{align*}
\tri\bsigma-\bsigma_h\tri_i^2 = a_i(\bsigma-\bsigma_h, \bsigma-\bsigma_h)
 = a_i(\bsigma-\bsigma_h, \bsigma-\btau_h)
 \leq \tri\bsigma-\bsigma_h\tri_i \tri\bsigma-\btau_h\tri_i.
\end{align*}
Thus $\tri\bsigma-\bsigma_h\tri_i \leq  \tri\bsigma-\btau_h\tri_i$.  Since $\btau_h$ is chosen arbitrarily, the theorem is proved.
\end{proof}

Define the following piecewise function spaces on the triangulation $\cT$,
 \begin{align*}
H^{s}(\cT)&=\{v\in L^2(\O):v|_K\in H^{s_K}(K) \,\ \forall K\in\cT \}, \\
H^s(\divvr; \cT)&=\{\btau \in (L^2(\O))^d: \btau|_K\in (H^{s_K}(K))^d, \gradt\btau|_K\in H^{s_K}(K)\,\ \forall K\in\cT \},
\end{align*}
with $s$ a piecewisely defined function, $s|_K=s_K>0$.

\begin{thm} \label{thm_LSFEM_apriori}
Assume the exact solution $\bsigma\in  H^s(\divvr; \cT)$, for $s>0$ defined piecewisely, and $\bsigma_h$ is the solution of the corresponding LSFEM \eqref{LSFEM} satisfying the corresponding condition on $\bbeta$ and $\gamma$ and the assumption on the boundary data, there exists a constant $C>0$ independent of the mesh size $h$, such that
\begin{align}\label{ine_apriori}
\tri \bsigma-\bsigma_h \tri_i
 \leq C
 \sum_{K\in \cT}h^{od_K}_K \left(\|\bsigma\|_{od_K,K}+\|\gradt\bsigma\|_{od_K,K}
 \right), \quad i=1 \mbox{ or } 2,
\end{align}
where $od_K = \min(k+1,s_K)$.
\end{thm}
\begin{proof}
By the definitions of $\tri\cdot\tri_i$ and the triangle inequality, it is easy to see that
\begin{eqnarray} \label{local1}
\tri \btau \tri_{1,K} \leq  \|\gradt \btau\|_{0,K} + (\|\gamma\|_{\infty,K}+d-1) \|\btau\|_{0,K},\\
\label{local2}
\mbox{and}\quad
\tri \btau \tri_{2,K} \leq  \| |\bbeta| \|_{\infty,K} \|\gradt \btau\|_{0,K} + \|\gamma\|_{\infty,K} \|\btau\|_{0,K}.
\end{eqnarray}
Then the inequality (\ref{ine_apriori}) is a direct consequence of Theorem \ref{thm_bestapp} and the approximation properties \eqref{rti} and \eqref{rti2}. 
\end{proof}

\begin{thm}
For the recovered $u_h$ by the first post-processing method \eqref{pp_u1} with $|\bbeta| >0$, we have
\begin{align} \label{erroru1}
\|u-u_h\|_0 \leq \|\bsigma-\bsigma_h\|_0. 
\end{align}
For the recovered $u_h$ by the second post-processing method \eqref{pp_u2} with $\gamma \neq 0$, 
we have
\begin{align} \label{erroru2}
\|u-u_h\|_0 \leq \dfrac{1}{\gamma}\|\gradt (\bsigma-\bsigma_h)\|_0.
\end{align}
\end{thm}
\begin{proof}
For the recovered $u_h$ by \eqref{pp_u1} with $|\bbeta| >0$, we have
\begin{align*}
	\|u-u_h\|_0 & = \|(\bbeta\cdot \bsigma-\bbeta\cdot \bsigma_h)/{|\bbeta|^2}\|_0	\leq \|\bsigma-\bsigma_h\|_0.
\end{align*}
For the recovered $u_h$ by \eqref{pp_u2} with $\gamma \neq 0$, \eqref{erroru2} is obvious.
\end{proof}

\begin{rem}
For both $\tri \cdot \tri_i$ norms, we have the following simple upper bound w.r.t. the standard $H(\divvr)$-norm:
$$
\tri \btau \tri_i \leq C \|\btau\|_{H(\divvr)}, \quad \forall \btau \in H_{0,-}(\divvr;\O),
$$
with $C$ depending on $\bbeta$ and $\gamma$ only.

On the other hand, our numerical test will also disprove the possibility of a coercivity w.r.t. the standard $H(\divvr)$-norm:
$$
\tri \btau \tri_i \geq C \|\btau\|_{H(\divvr;\O)}, \quad \forall \btau \in H_{0,-}(\divvr;\O),
$$
or a weak discrete version with an $h$-independent $C>0$,
$$
\tri \btau \tri_i \geq C \|\btau\|_{H(\divvr;\O)}, \quad \forall \btau \in RT_{k,0,-}.
$$
Because if one of such coercivity results holds, we will have
$$
\|u-u_h\|_0  \leq C \|\bsigma -\bsigma_h\|_{H(\divvr)} \leq C \tri \bsigma -\bsigma_h \tri_i.
$$
For a uniform mesh and an $RT_0$ approximation of $\bsigma_h$, we should have
$$
\|u-u_h\|_0 \leq C \|\bsigma -\bsigma_h\|_{H(\divvr)} \leq  C h,
$$
for piecewise smooth solutions with a discontinuity aligned mesh or a global smooth solution with a Peterson's mesh, which is not the case in our numerical examples 5.2, 5.3, and 5.4. 
\end{rem}

\subsection{A posteriori error estimation}
The least-squares functionals can be used to define the following fully computable a posteriori  local indicators and global error estimators:
\begin{eqnarray*}
\eta_{1,K} &=&
\left(\|\gradt \bsigma_h+\tilde{\gamma}(\bbeta \cdot \bsigma_h)- f\|_{0}^2 + \sum_{i=1}^{d-1}\|\bsigma_h \cdot \bbeta_{\bot}^{(i)} \|_{0}^2\right)^{1/2}, \quad \forall\, K\in \cT,\\[2mm]
\eta_{2,K} &=& \|\gamma \bsigma_h + \bbeta (\gradt\bsigma_h -f) \|_{0,K}, \quad \forall\, K\in \cT,
\end{eqnarray*}
and
$$
\eta_{i} := \left( \sum_{K\in\cT}\eta_{i,K}^2 \right)^{1/2} = J_{i}(\bsigma_h;f,g)^{1/2}, \quad i= 1 \mbox{ or }2.
$$

\begin{thm}
The a posteriori error estimator $\eta_i$ is exact with respect to the least-squares norm $\tri \cdot\tri_i$:
$$
\eta_i = \tri \bsigma-\bsigma_h \tri_i \quad \mbox{and} \quad \eta_{i,K} = \tri \bsigma-\bsigma_h \tri_{i,K}.
$$
The following local efficiency bound is also true with $C>0$ independent of the mesh size $h$:
$$
	\eta_{i,K} \leq C \|\bsigma-\bsigma_h\|_{H(\divvr;K)}, \quad \forall K\in\cT.
$$
\end{thm}
\begin{proof}
For LSFEM1, using the facts $f = \gradt \bsigma+\tilde{\gamma}(\bbeta \cdot \bsigma)$ and $\bsigma \cdot \bbeta_{\bot}^{(i)}=0, \forall i \in I$, for the exact solution $\bsigma$, we get
\begin{eqnarray*}
\eta_1^2 &=&
\|\gradt \bsigma_h+\tilde{\gamma} (\bbeta \cdot \bsigma_h)- f\|_{0}^2 + \sum_{i=1}^{d-1}\|\bsigma_h \cdot \bbeta_{\bot}^{(i)} \|_{0}^2\\
&=& \|\gradt (\bsigma-\bsigma_h)+\tilde{\gamma}\bbeta \cdot (\bsigma-\bsigma_h)\|_{0}^2 + \sum_{i=1}^{d-1}\|(\bsigma-\bsigma_h) \cdot \bbeta_{\bot}^{(i)} \|_{0}^2\\
&=&\tri \bsigma-\bsigma_h \tri_1^2.
\end{eqnarray*}
The proofs of LSFEM2 and the local exactness are identical.

With local bounds \eqref{local1} and \eqref{local2}, 
the local efficiency bound for the standard $H(\divvr)$ norm can be easily proved.
\end{proof}

\begin{rem}
Due to the fact that the least-squares functional norm is not equivalent to the standard $H(\divvr)$ norm, it is
impossible to get the corresponding reliability result w.r.t. the $H(\divvr)$ norm.
\end{rem}

\section{Computational Examples}
In all our numerical examples, the lowest order element $RT_0$ is used to approximate the flux $\bsigma$. In the adaptive mesh refinement algorithm,  the D\"ofler's bulk marking strategy with $\theta =0.5$ is used and  the algorithm is stopped when the total number of DOFs reaches $10^6$. All refinements are based on the longest edge bisection algorithm. For all the numerical examples with domain $(0,1)^2$, the mesh shown in Fig. \ref{initialmesh} is used as an initial mesh.
\begin{figure}[!ht]
    \centering
   \begin{minipage}[!hbp]{0.45\linewidth}
        \includegraphics[width=0.99\textwidth,angle=0]{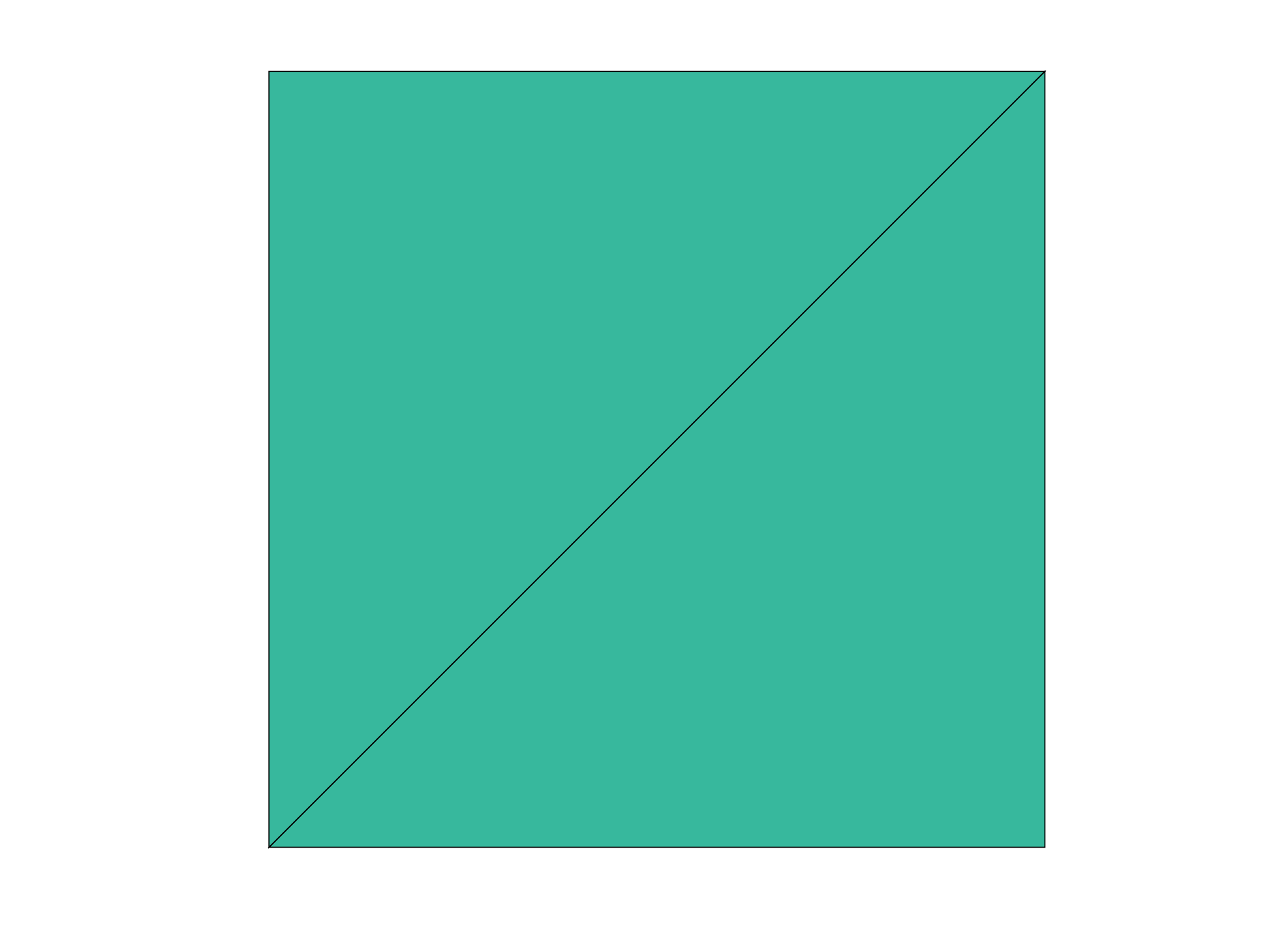}
        \end{minipage}%
        \caption{Initial mesh for all examples with a $(0,1)^2$ domain}%
        \label{initialmesh}
\end{figure}

Since we have two formulations for $\bsigma_h$ and two recoveries for $u_h$, there are four variants to find the numerical solution pair $(u_h,\bsigma_h)$. We use LSFEMi-j to denote them, with $i$ denoting the method to get $\bsigma_h$ and $j$ denoting the method to recover $u_h$. For example, LSFEM1-2 means we use LSFEM1 to get $\bsigma_h$ and use the second recovery to get $u_h$. For almost all our numerical tests, we find that the results of the four combinations are identical. We only show the full 4 combinations for one example, and show only one option for other examples.


\subsection{An example with a constant advection field and a piecewise constant solution, on a matching grid}
In this example, we only need to do a thought experiment, although the actual computation does confirm our result.

Consider the following problem: $\O = (0,1)^2$ with $\bbeta = (1/\sqrt{2},1/\sqrt{2})^T$. The inflow boundary is $\{x=0, y\in (0,1)\} \cup \{x\in (0,1), y=0\}$, i.e., the west and south boundaries of the domain.
Let $\gamma =1$ and choose $g$ and $f$ such that the exact solution $u$ is
$$
u = \left\{ \begin{array}{lll}
1 & \mbox{in} & y>x, \\[2mm]
0 & \mbox{in} & y<x.
\end{array} \right.
$$
If we choose the mesh aligned with the discontinuity, for example, any refinements of the mesh in Fig. \ref{initialmesh}. Note that the exact $\bsigma \in RT_0$. By the best approximation property Theorem \ref{thm_bestapp}, the numerical solution $\bsigma_h$ is identical to the exact solution. The recovered $u_h$ is also identical to $u$. So no further refinements are needed.
This is not true when $C^0$ finite elements are used to approximate the discontinuous $u$ as in \cite{BochevChoi:01,DMMO:04,BG:09,BG:16}, many unnecessary refinements are needed.

\subsection{An example with a global smooth solution}
Consider the following simple problem: $\O = (0,1)^2$ with $\bbeta = (1,1)^T$. The inflow boundary is $\{x=0, y\in (0,1)\} \cup \{x\in (0,1), y=0\}$, i.e., the west and south boundaries of the domain. Let $\gamma=1$. Choose $f$ and $g$ such that the exact solution is $u =\sin(x+y)$.

\begin{figure}[!htb]
\centering 
\subfigure[LSFEM1-1]{ 
\includegraphics[width=0.45\linewidth]{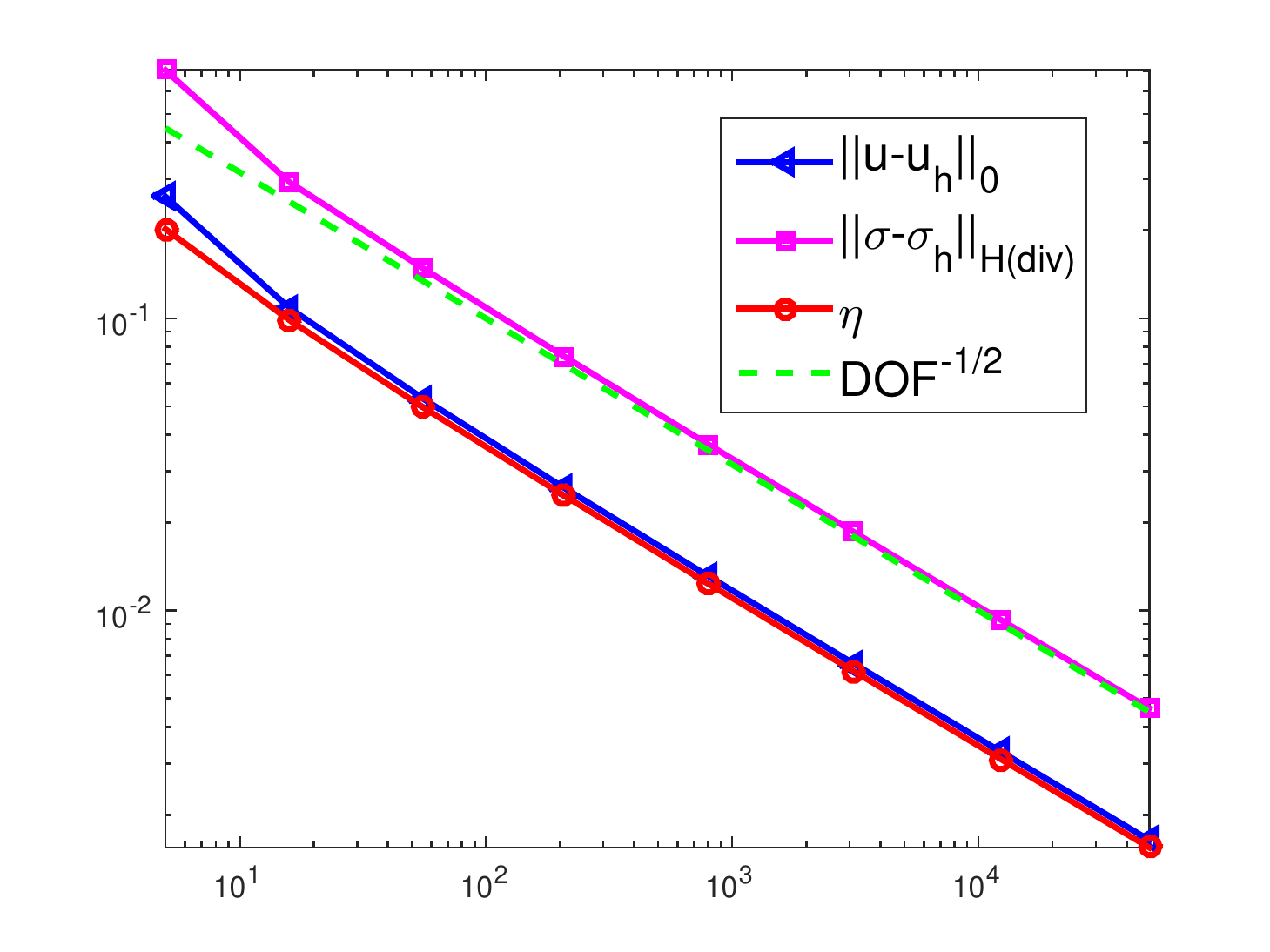}}
\hspace{0.01\linewidth}
\subfigure[LSFEM1-2]{
\includegraphics[width=0.45\linewidth]{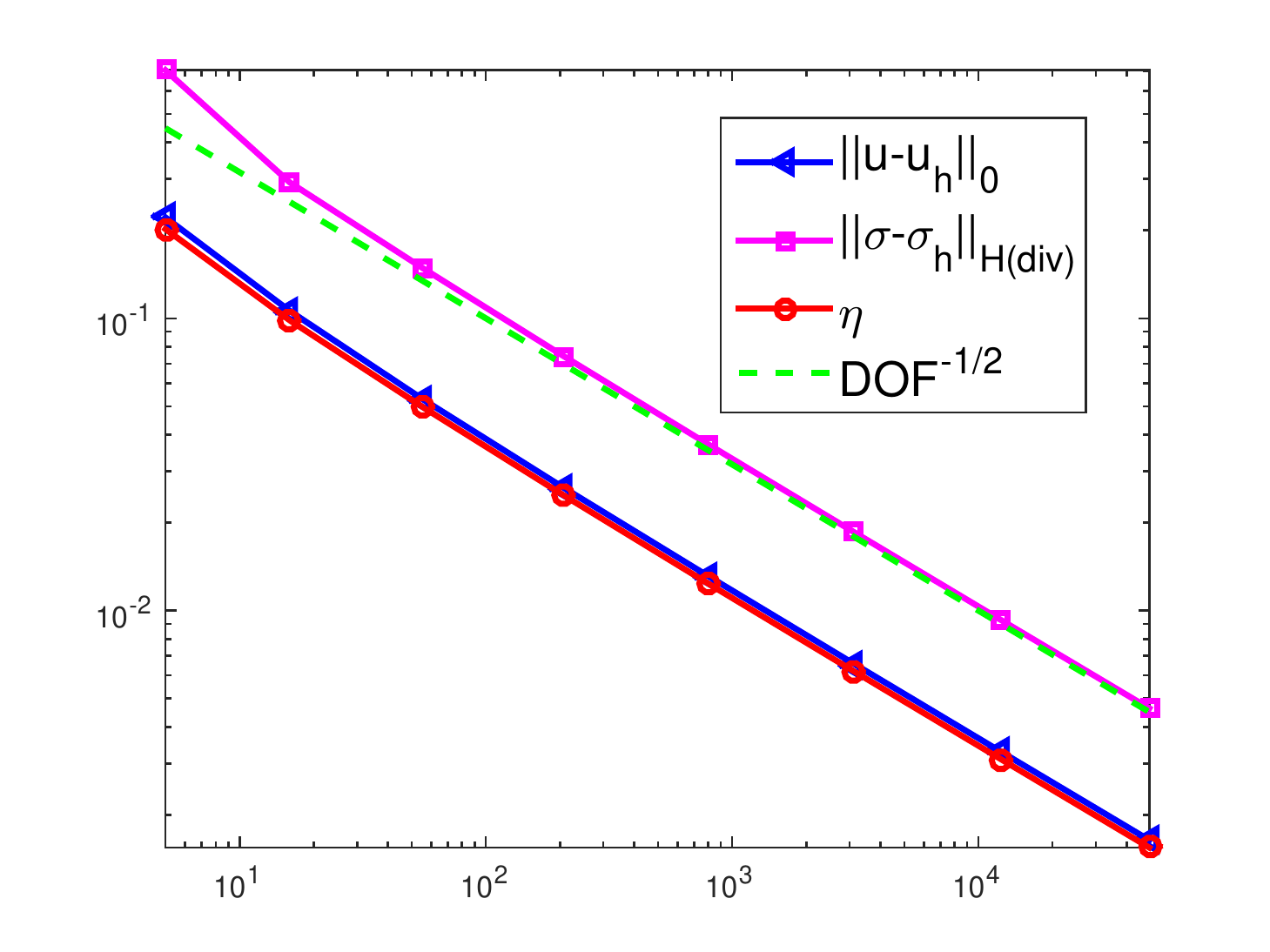}}
\vfill
\subfigure[LSFEM2-1]{
\includegraphics[width=0.45\linewidth]{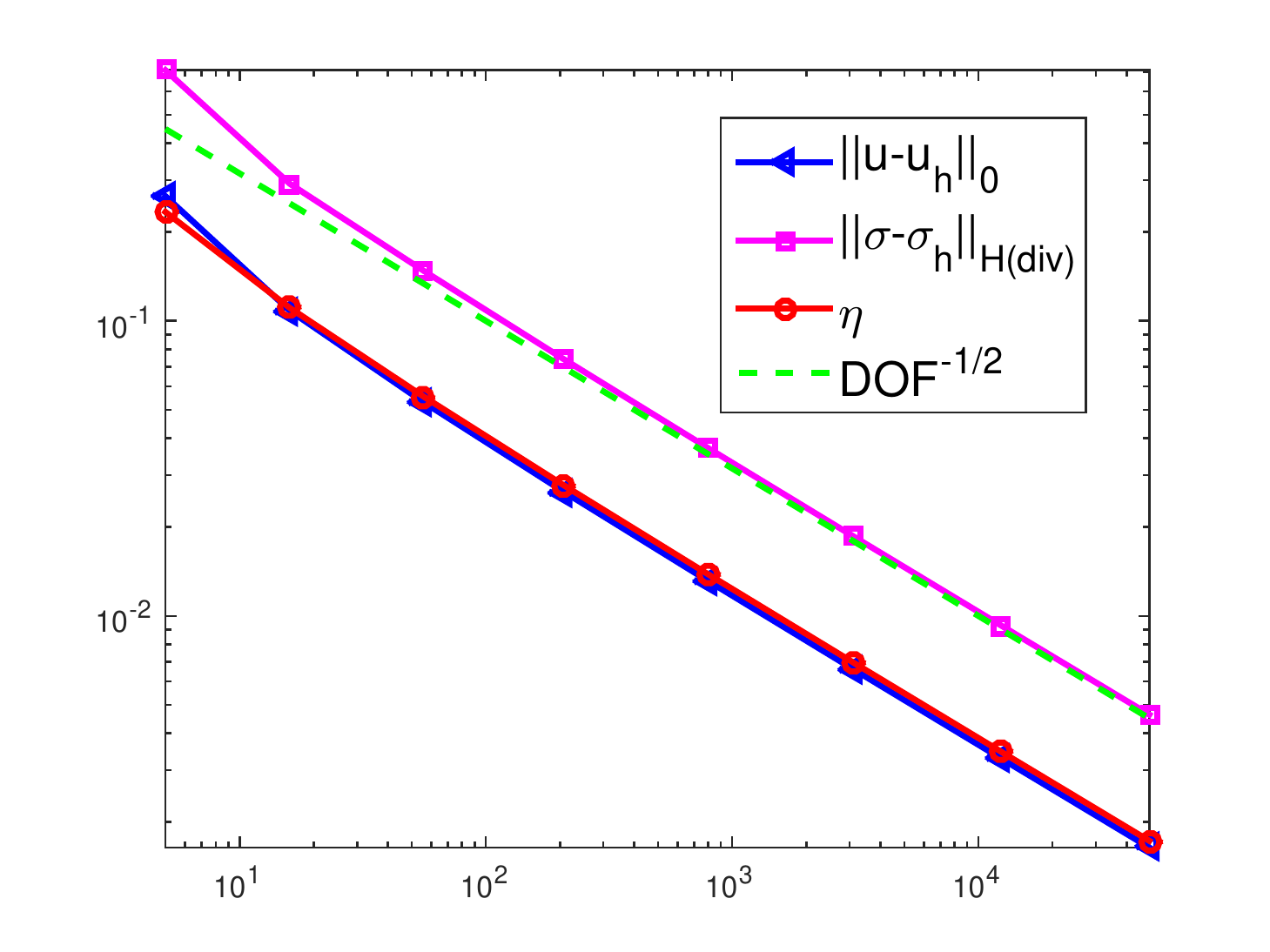}}
\hspace{0.01\linewidth}
\subfigure[LSFEM2-2]{
\includegraphics[width=0.45\linewidth]{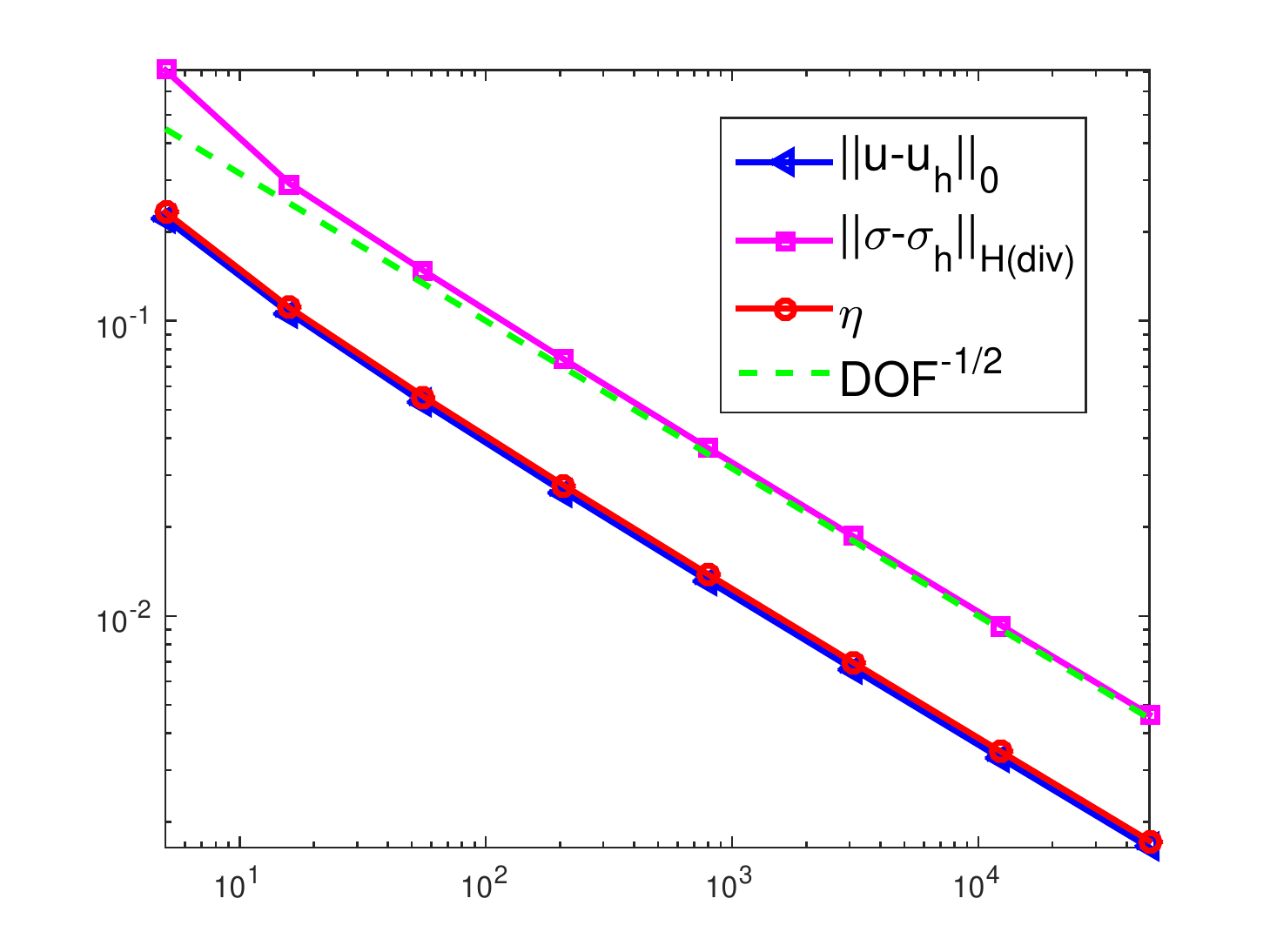}}
\caption{Global smooth solution: convergence histories on uniformly refined meshes}
\label{error_smooth}
\end{figure}

In Fig. \ref{error_smooth}, the convergence histories on uniformly refined meshes are shown. For all 4 methods, errors measured in the least-squares energy norms, $\|\bsigma-\bsigma_h\|_{H(\divvr)}$, and $\|u-u_h\|_0$ are all of order $1$. 

In Fig. \ref{error_smooth},  we also notice that all combinations of the methods have almost identical numerical results.

\subsection{Peterson example}
We consider a famous example from  Peterson \cite{Peterson:91}. Consider the test problem from section 3 of Peterson \cite{Peterson:91}: Let $\O = (0,1)^2$ and $\bbeta = (0,1)^T$. The inflow boundary $\Gamma_-$ is $\{x\in (0,1), y=0\}$, i.e., the south boundary of the domain.
\begin{eqnarray}
u_y = \gradt(\bbeta u)&=& 0 \quad \mbox{ in } \O,\\
u|_{\Gamma_-} &=& x \quad \mbox{ on } \Gamma_-.
\end{eqnarray} 
The exact solution is $u=x$. The mesh is chosen to be in the pattern on the left of Fig. \ref{error_peterson}. Since $\gamma=0$, we only test this example with the method LSFEM1-1. We compute a series of solutions by LSFEM1-1 on meshes with $h$ from $1/6$, $1/12$, $\cdots$, to $1/768$. The convergence result is plotted on the right of Fig. \ref{error_peterson}. It is observed that the error in LS norm still converges in the order of $1$, but the $L^2$-norm of $u-u_h$ and $H(\divvr)$-norm of $\bsigma-\bsigma_h$ only converges in the order of $3/4$.

\begin{figure}[!htb]
\centering 
\subfigure[Peterson mesh with $h=1/6$]{ 
\includegraphics[width=0.45\linewidth]{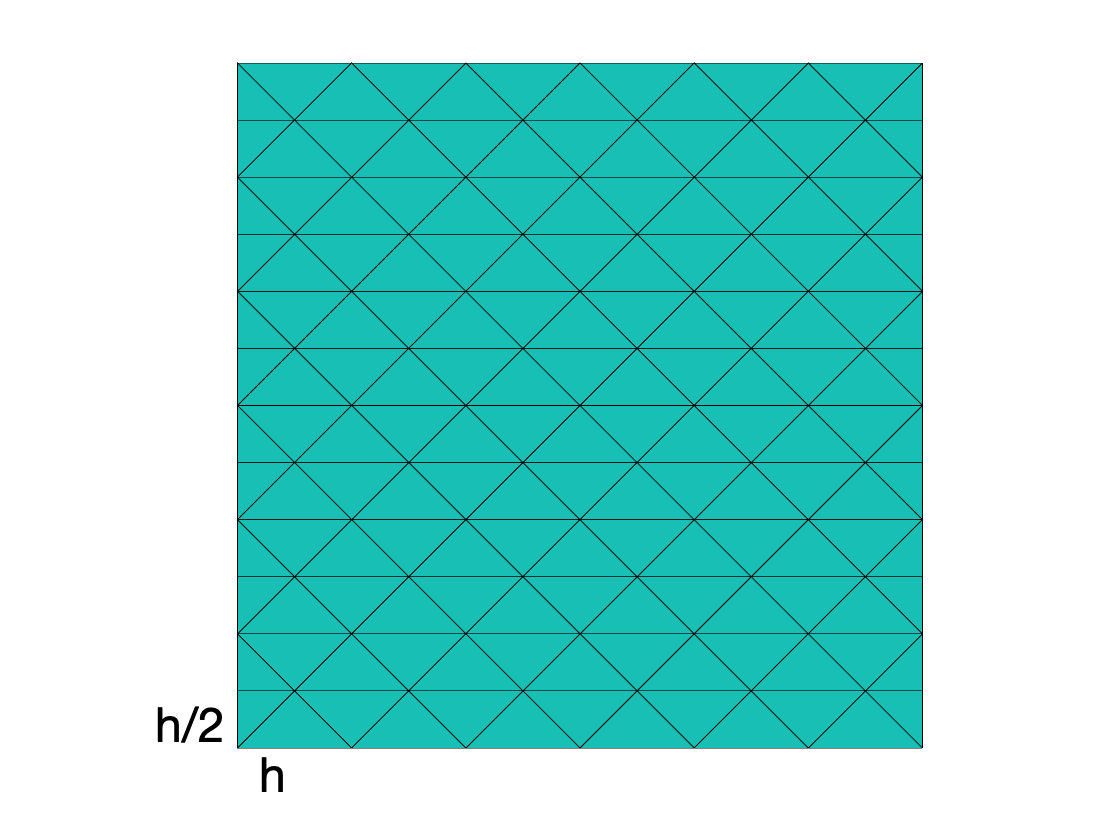}}
~
\subfigure[LSFEM convergence]{
\includegraphics[width=0.45\linewidth]{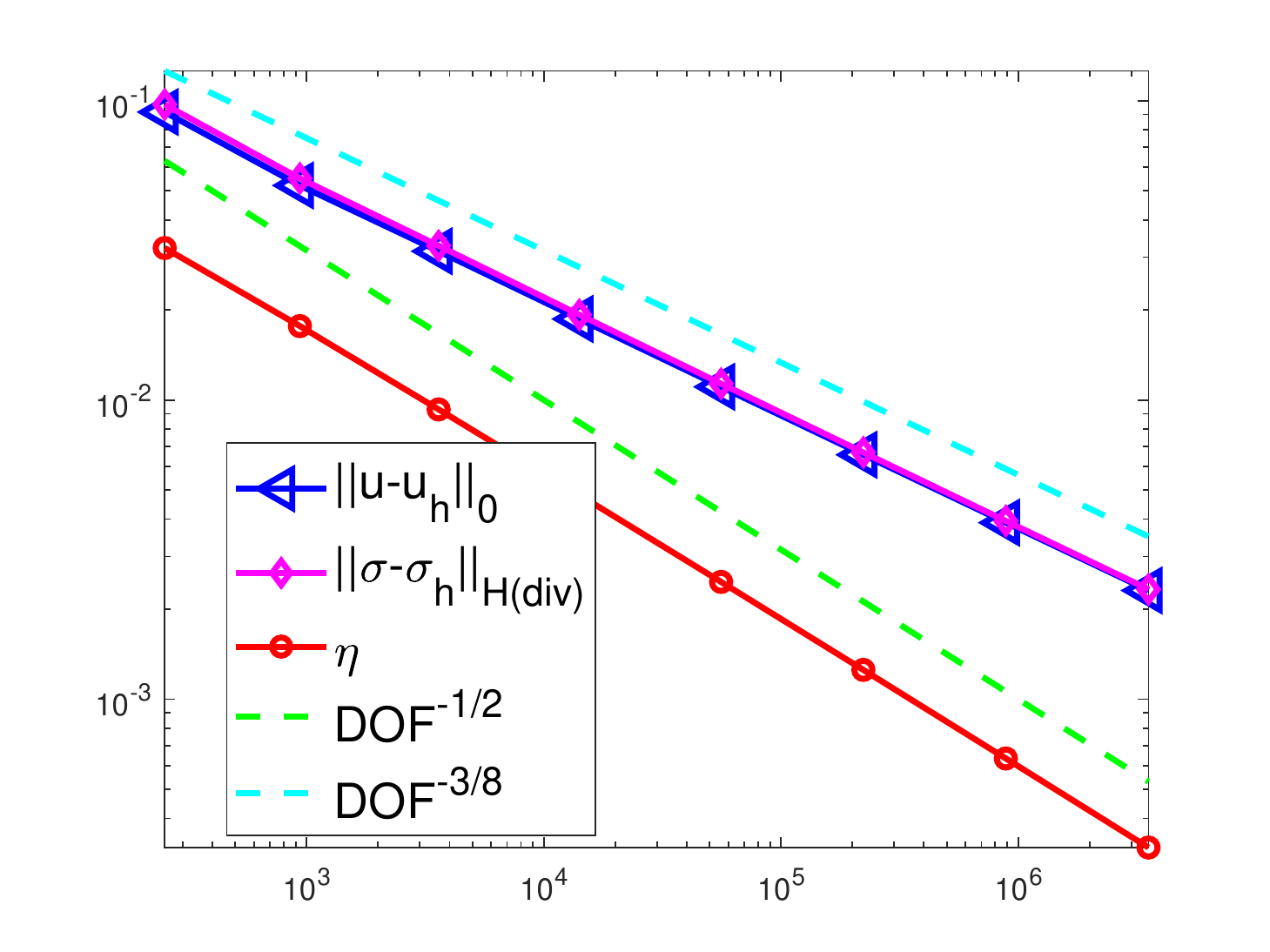}}
\caption{Peterson example}
 \label{error_peterson}
\end{figure}

\subsection{An example with a piecewise smooth solution, on a matching grid}
Consider the following problem: $\O = (0,1)^2$ with $\bbeta = (1/\sqrt{2},1/\sqrt{2})^T$. The inflow boundary is $\{x=0, y\in (0,1)\} \cup \{x\in (0,1), y=0\}$, i.e., the west and south boundaries of the domain. Let $\gamma =1$. Choose $f$ and $g$ such that the exact solution $u$ is
$$
u = \left\{ \begin{array}{lll}
	\sin(x+y) &\mbox{if} &y>x, \\[2mm]
	\cos(x+y) &\mbox{if} &y<x.
\end{array} \right.
$$

We choose an initial mesh that matches the discontinuity (Fig. \ref{initialmesh}) and uniformly refine it for $8$ times. In Fig. \ref{error_pws}, we show the convergence histories. The convergence order of the error in least-squares norms is $1$, which matches the optimal convergence theory. The orders of $\|\bsigma-\bsigma_h\|_{H(\divvr)}$ and $\|u-u_h\|_0$ are less than $1$ (about $0.6$ at late stages). 

This one and example 5.3 suggest that the norm equivalence (or in discrete sub-spaces):
$$
\tri \btau \tri \approx \|\btau\|_{H(\divvr;\O)} \quad \forall \btau \in H_{0,-}(\divvr;\O),
$$
should not be true.
\begin{figure}[!ht]
\centering 
\includegraphics[width=0.45\linewidth]{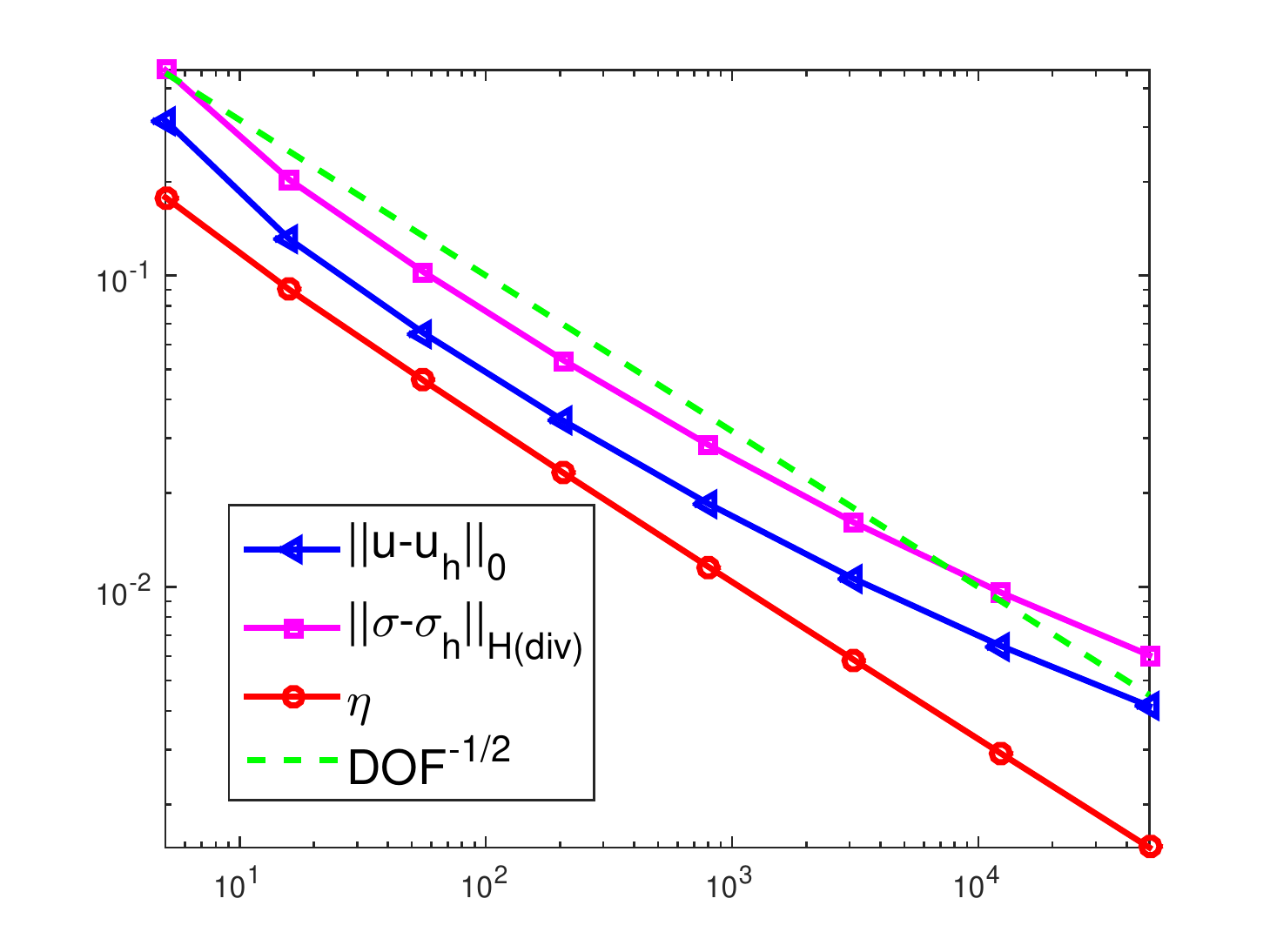}
\caption{Piecewise smooth solution with a matching mesh: convergence histories for on uniformly refined meshes}
\label{error_pws}
\end{figure}

\subsection{An example with a piecewise smooth solution, on non-matching grid}
In this example, we discuss the over/undershooting of the solution when the mesh is not matched with discontinuity.

Consider the problem: $\O = (0,2)\times (0,1)$ with $\bbeta = (0,1)^T$. The inflow boundary is $\{x\in (0,2), y=0\}$, i.e., the south boundary of the domain. Let $\gamma =1$ and $f=1$. Choose the inflow boundary condition
$$
u(x,0) = \left\{ \begin{array}{lll}
	0 & \mbox{if} & x< \pi/3, \\[2mm]
	1 & \mbox{if} & x > \pi/3,
\end{array} \right.
$$
such that the exact solution is
$$
u(x,y) = \left\{ \begin{array}{lll}
1-e^{-y} & \mbox{if} & x< \pi/3, \\[2mm]
1 & \mbox{if} & x > \pi/3.
\end{array} \right.
$$
We set the initial mesh to be as shown in Fig. \ref{inimesh_pi}. The bottom central node is $(\pi/3,0)$ and the top central node is $(1,1)$. So the inflow boundary mesh is matched with the inflow boundary discontinuity while the mesh is not aligned with the discontinuity in general and will never match it if a bisection mesh refinement is used. 
\begin{figure}[!ht]
   \centering
   \begin{minipage}[!hbp]{0.45\linewidth}
        \includegraphics[width=0.99\textwidth,angle=0]{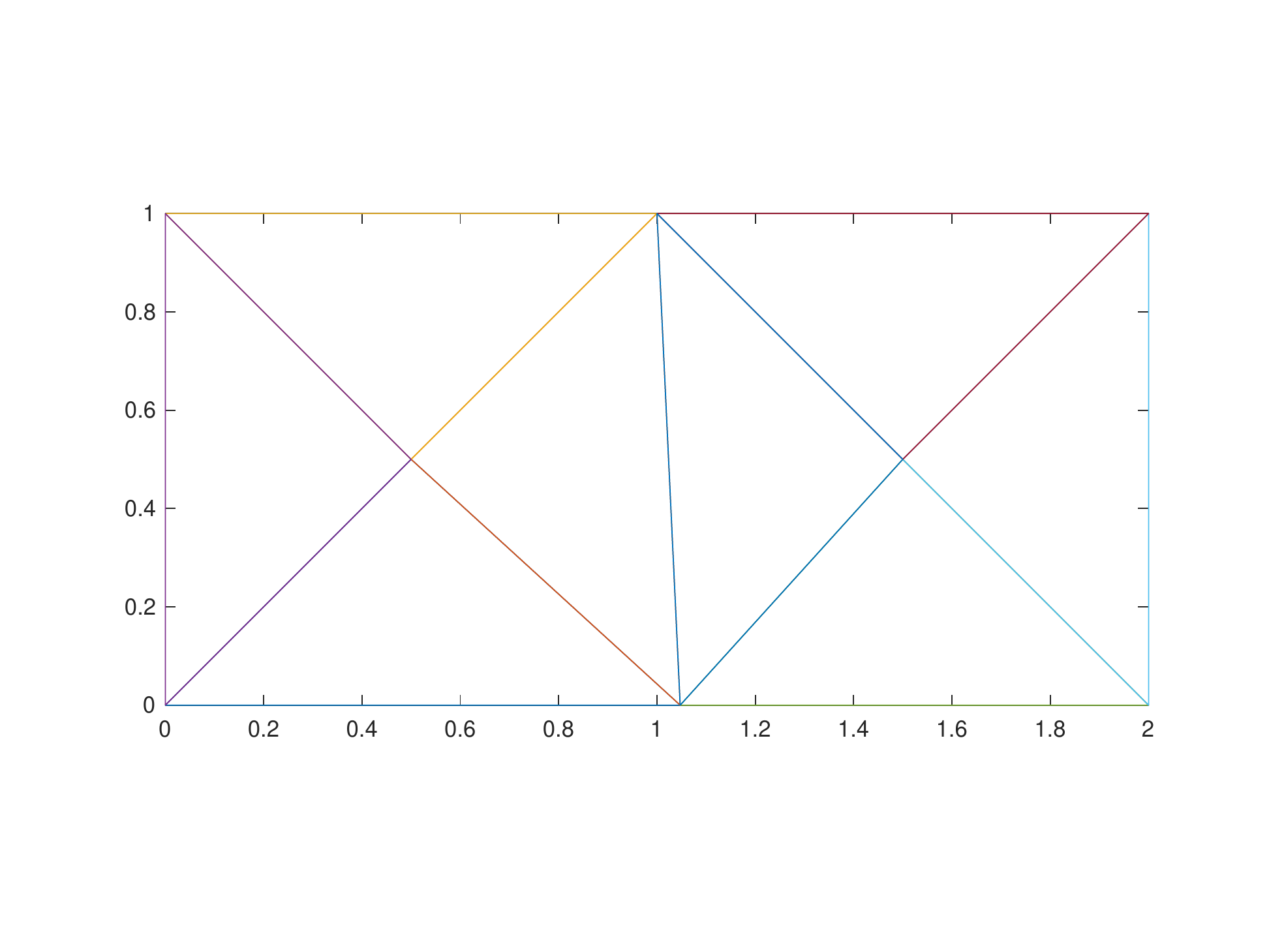}
       \end{minipage}%
       \caption{Piecewise continuous problem on a non-matching mesh: the initial mesh}%
        \label{inimesh_pi}
\end{figure}

On the left of Fig. \ref{outflow_solution_pws}, we show the solution $u_h$ on a mesh after 8 uniform refinements of the initial mesh. In order to study the overshooting phenomenon on the outflow boundary, we only draw the graph of $u_h$ on $y=1$, that is, we plot the $u_h$ value at the midpoint of x-axis of each elements with edges on $y=1$. The solutions are almost identical for all 4 combinations. Small under/overshooting can be found. The maximum of numerical $u_h$ on $y=1$ is $1.0430$ with the exact solution $u=1$, and the minimum of numerical $u_h$ is $0.6210$ with the exact $u=0.6321$.

\begin{figure}[!htb]
\centering 
\subfigure[outflow solution]{ 
\includegraphics[width=0.45\linewidth]{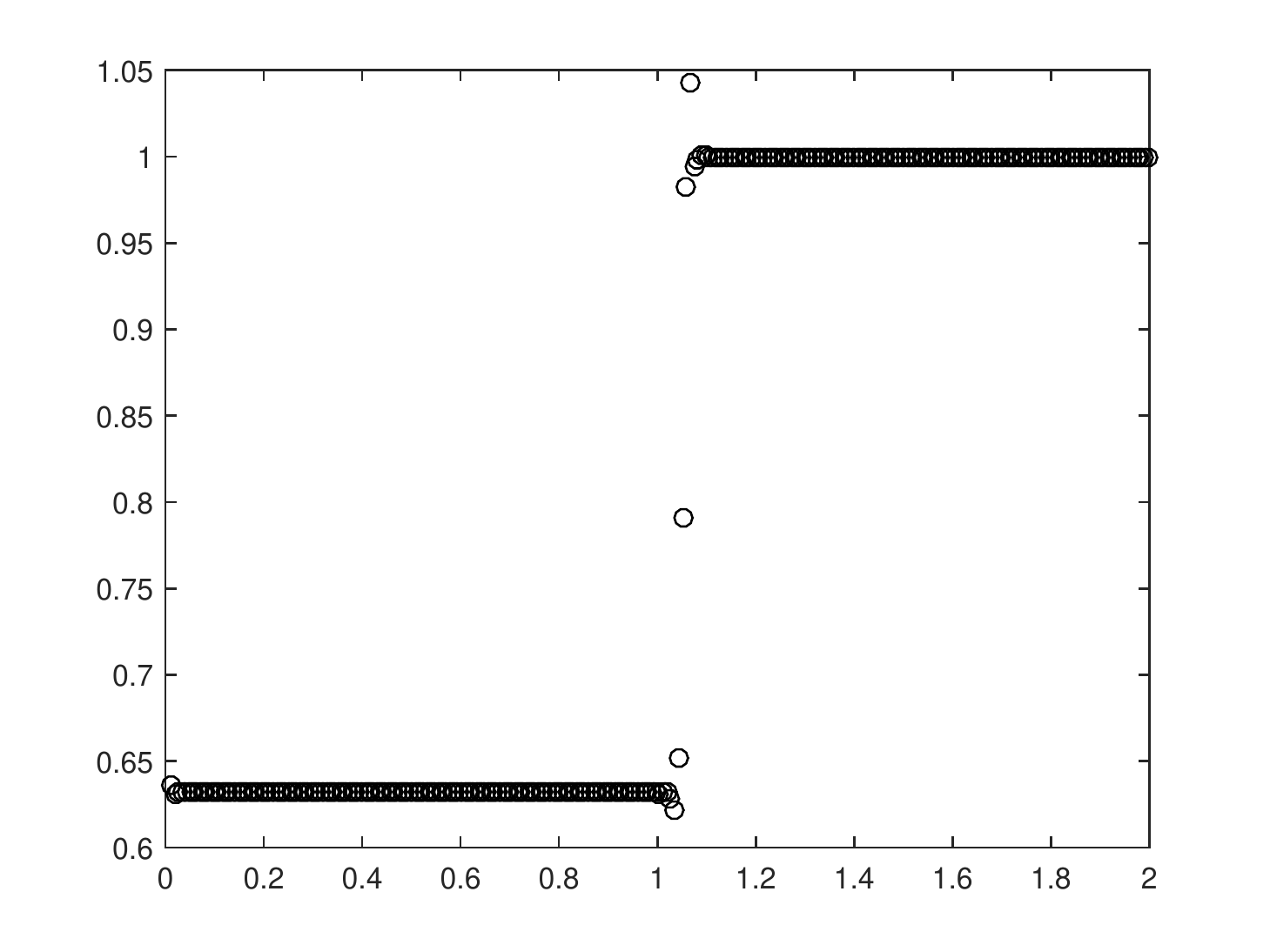}}
\hspace{0.01\linewidth}
\subfigure[convergence histories]{
\includegraphics[width=0.45\linewidth]{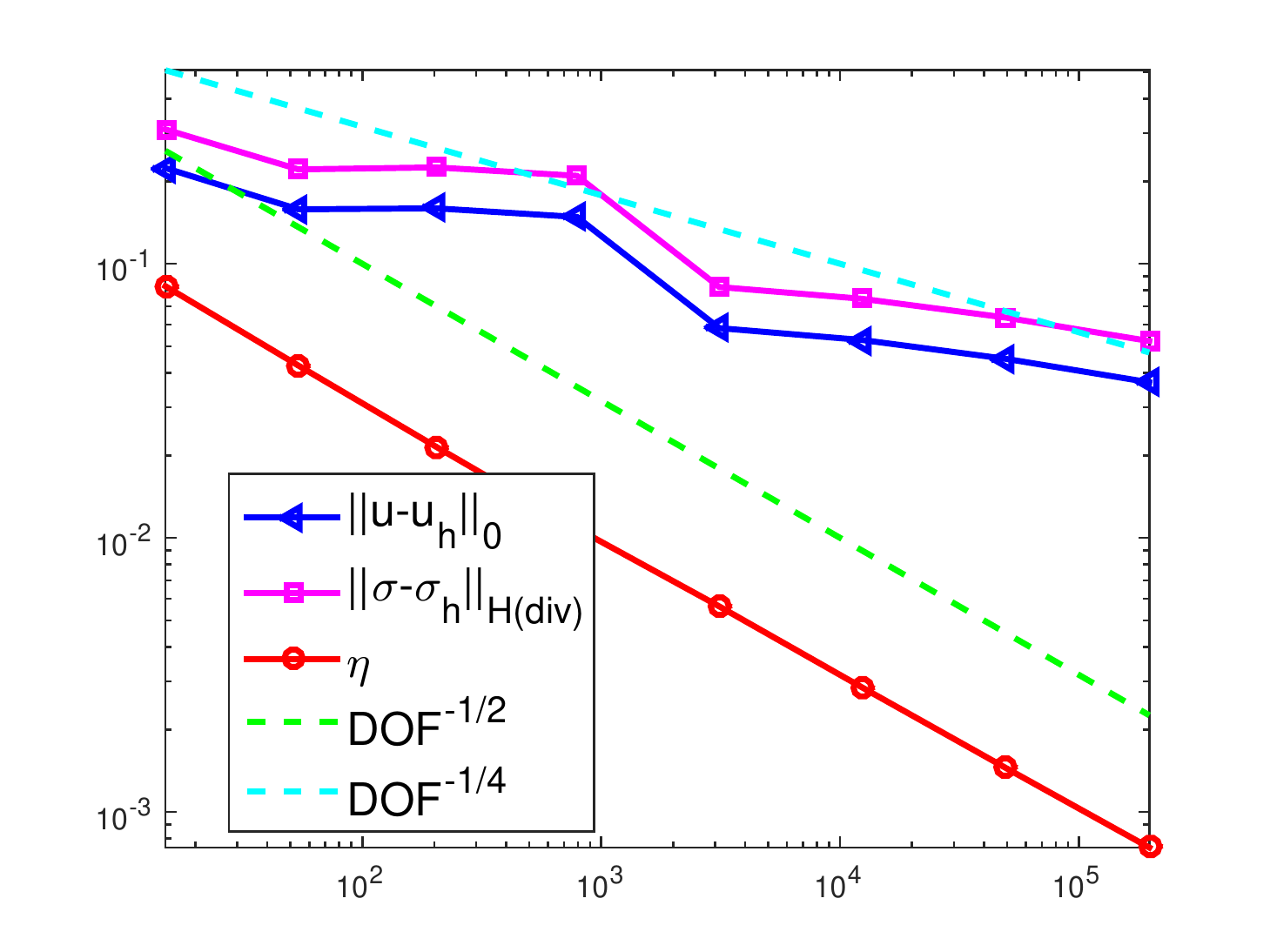}}
\caption{Piecewise continuous problem on a non-matching uniform mesh}
\label{outflow_solution_pws}
\end{figure}

On the right of Fig. \ref{outflow_solution_pws}, we plot the convergence results of uniform refinements. The order of convergence of errors measured in least-squares norms is still about $1$, the optimal order. The order of $\|u-u_h\|_0$ is smaller than $1/2$.

We then test the problem with adaptive algorithm. In Fig. \ref{adaptivemesh_pwc}, adaptively refined meshes after several iterations are shown. Clearly, more refinements are required along the discontinuity. We do find that for both two LSFEMs, the meshes are almost identical.

\begin{figure}[!ht]
\centering 
\subfigure[LSFEM 1]{ 
\includegraphics[width=0.45\linewidth]{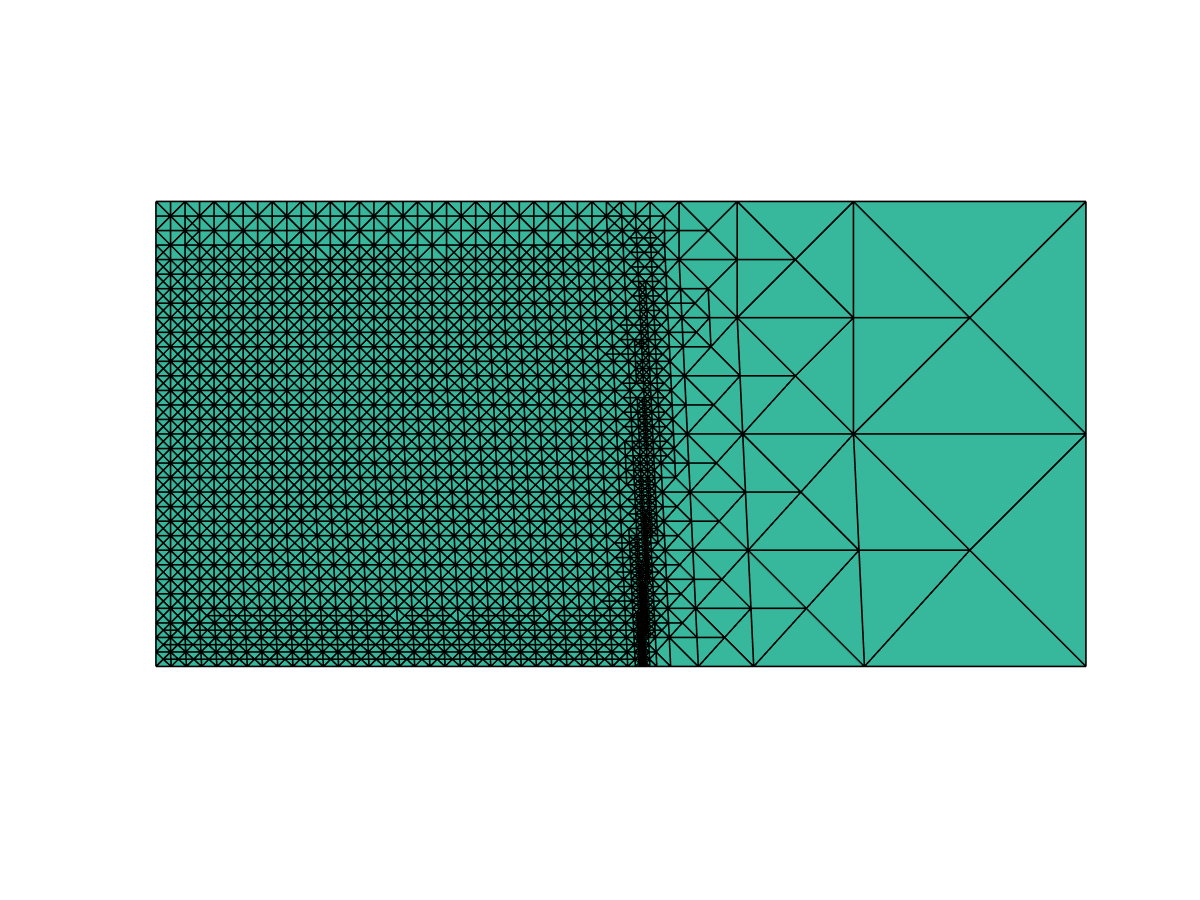}}
\hspace{0.01\linewidth}
\subfigure[LSFEM 2]{
\includegraphics[width=0.45\linewidth]{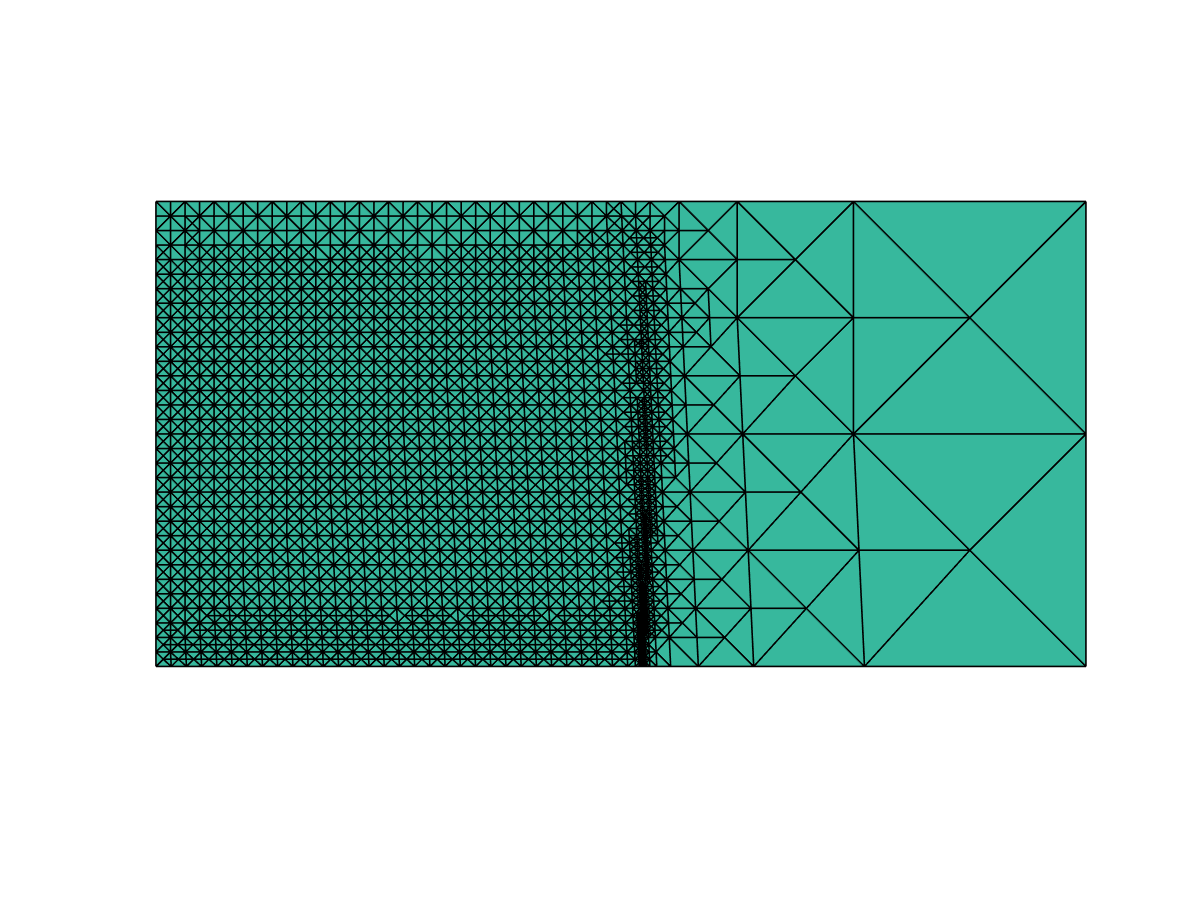}}
\caption{Piecewise continuous problem on a non-matching mesh: adaptively refined meshes after several iterations}
\label{adaptivemesh_pwc}
\end{figure}

In Fig. \ref{error_adaptive_pws}, the convergence histories for the adaptive meshes are shown. The convergence order of the error measured in the least-squares norm is optimal with order $1$, while the order of $\|u-u_h\|_0$ is about $1/2$,  which is the same order (also the best possible order) as uniform refinements for discontinuous solution on an aligned mesh, Example 5.4.

\begin{figure}[!ht]
\centering 
\includegraphics[width=0.45\linewidth]{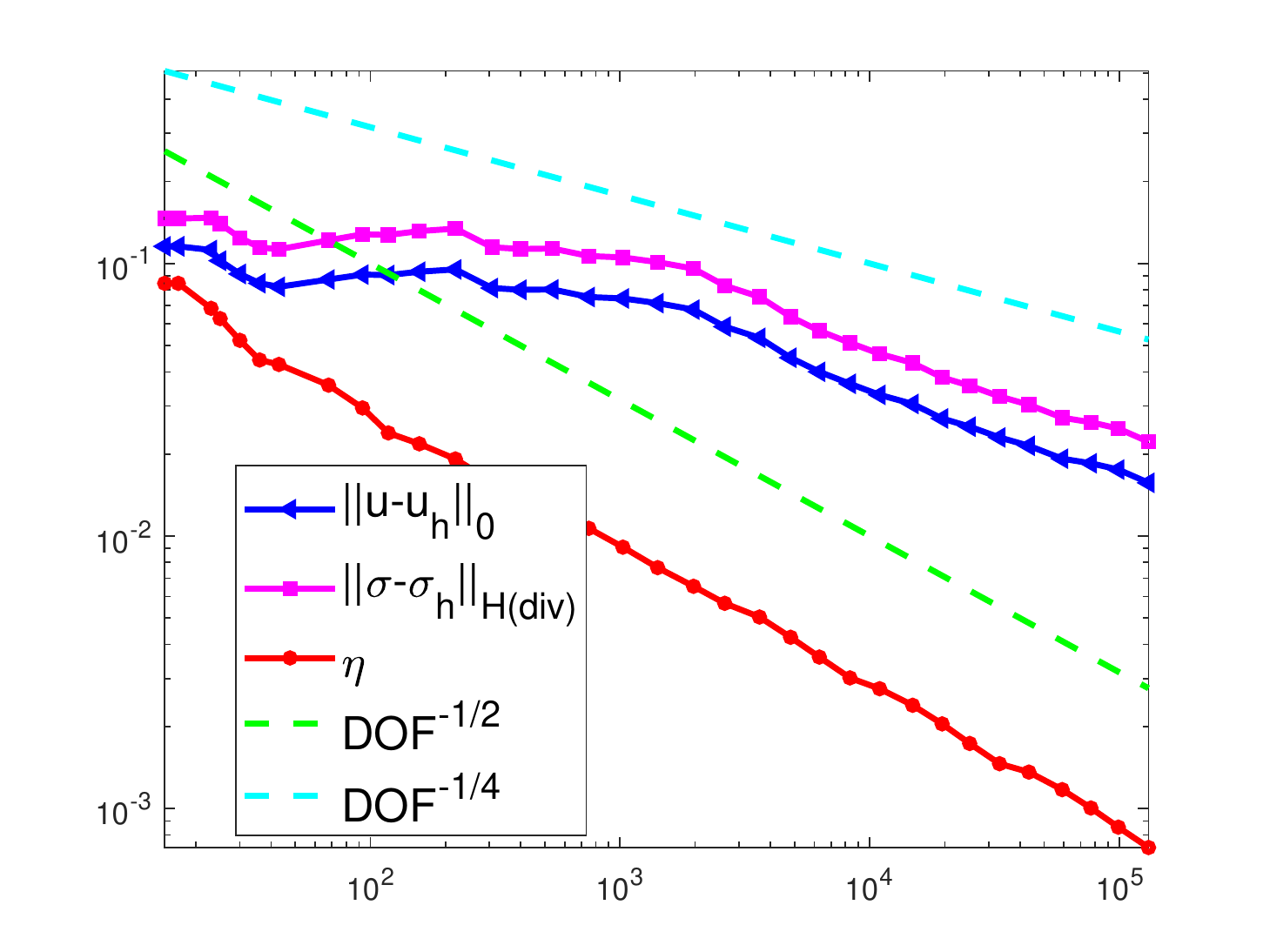}
\caption{Piecewise continuous problem on a non-matching mesh: convergence histories on adaptive refined meshes} 
\label{error_adaptive_pws}
\end{figure}

On the left of Fig. \ref{overshooting_pws}, we show the decreasing of the overshooting values by adaptive mesh refinements. Here, the overshooting value is defined as $\max(\max(u_h-1), \min(1-e^{-1}-u_h))$ on $y=1$, still along the outflow boundary. We clearly see that the overshooting value begins to decrease after the mesh is reasonably fine. When the mesh is very coarse, the overshooting is actually not very severe since $u$ is approximated by $P_0$.

\begin{figure}[!ht]
\centering 
\subfigure[reduction of overshooting]{ 
\includegraphics[width=0.45\linewidth]{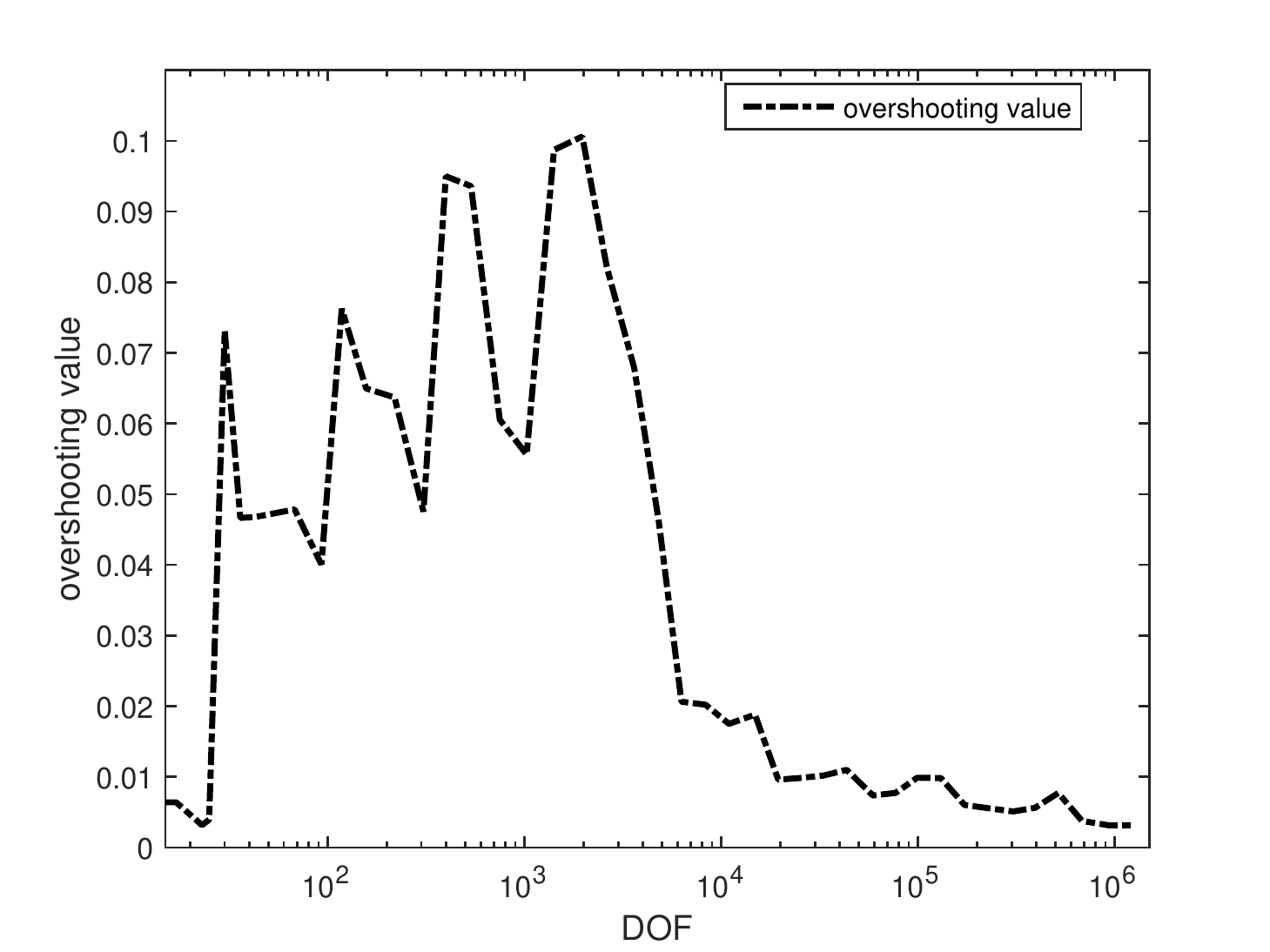}}
\hspace{0.01\linewidth}
\subfigure[outflow solution]{
\includegraphics[width=0.45\linewidth]{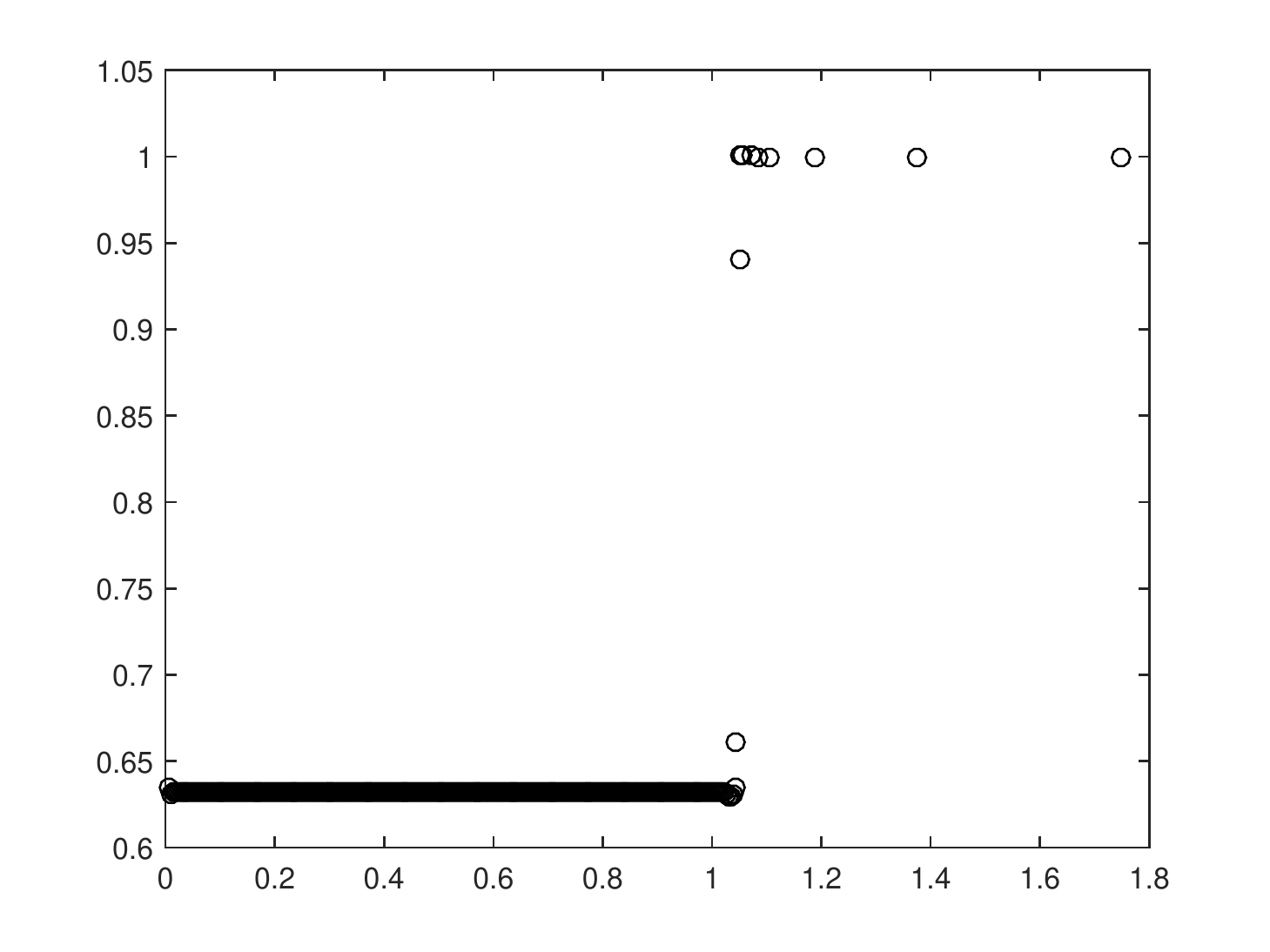}}
\caption{Piecewise continuous problem on a non-matching adaptive mesh}
\label{overshooting_pws}
\end{figure}

On the right of Fig. \ref{overshooting_pws}, we show the numerical solution $u_h$ on the outflow boundary for the final mesh. It is clear that when with the adaptive mesh and $P_0$ recovery, the overshooting phenomenon is almost neglectable. 


\subsection{An example with a piecewise smooth solution on a non-matching grid}
Consider the following simple problem: $\bbeta = (\sin(1/8),\cos(1/8))^T$ and $\O = (0,1)^2$. The inflow boundary is $\{x=0, y\in (0,1)\} \cup \{x\in (0,1), y=0\}$, i.e., the west and south boundaries of the domain. Let $\gamma =1$. Choose $f$ and $g$ such that the exact solution $u$ is
$$
u = \left\{ \begin{array}{lll}
	\sin(x+y) &\mbox{if}& y>\tan(1/8)x, \\[2mm]
	\cos(x+y) &\mbox{if}& y< \tan(1/8)x.
\end{array} \right.
$$
Note that with an initial mesh as in Fig. \ref{initialmesh}, any refinement of it will never match the discontinuity.

We show the convergence results after uniform refinements in Fig. \ref{error_uniform_pws}. The convergence order in the least-square norm is about $0.8$. Similar to example 5.5, it is worse than order $1$ but better than order $1/2$. The convergence order for $\|u-u_h\|_0$ is about $0.3$.

\begin{figure}[!ht]
\centering 
\subfigure{
\includegraphics[width=0.45\linewidth]{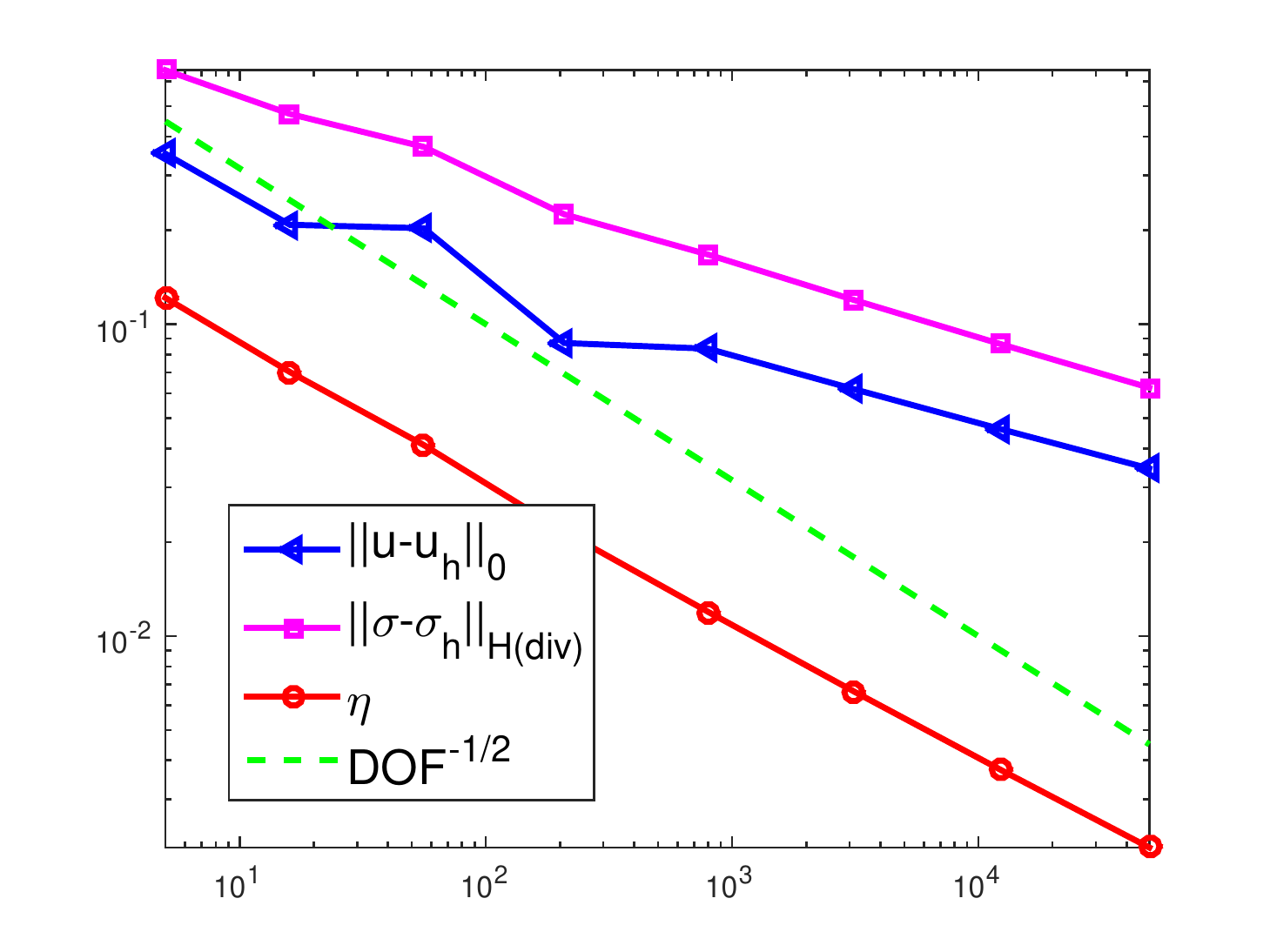}}
\caption{Piecewise smooth solution on a non-matching grid: convergence histories on uniformly refined meshes}
\label{error_uniform_pws}
\end{figure}

On the left of Fig. \ref{adaptive_mesh_pws}, adaptively refined meshes after several iterations are shown. Many refinements are generated near the discontinuity. On the right of Fig. \ref{adaptive_mesh_pws}, convergence histories of the adaptive LSFEMs are shown. The rate of convergence of the error in least-squares norms is about order $1$, and the order of $\|u-u_h\|_0$ is about $0.5$.

\begin{figure}[!ht]
\centering 
\subfigure[adaptively refined meshes after several iterations]{ 
\includegraphics[width=0.45\linewidth]{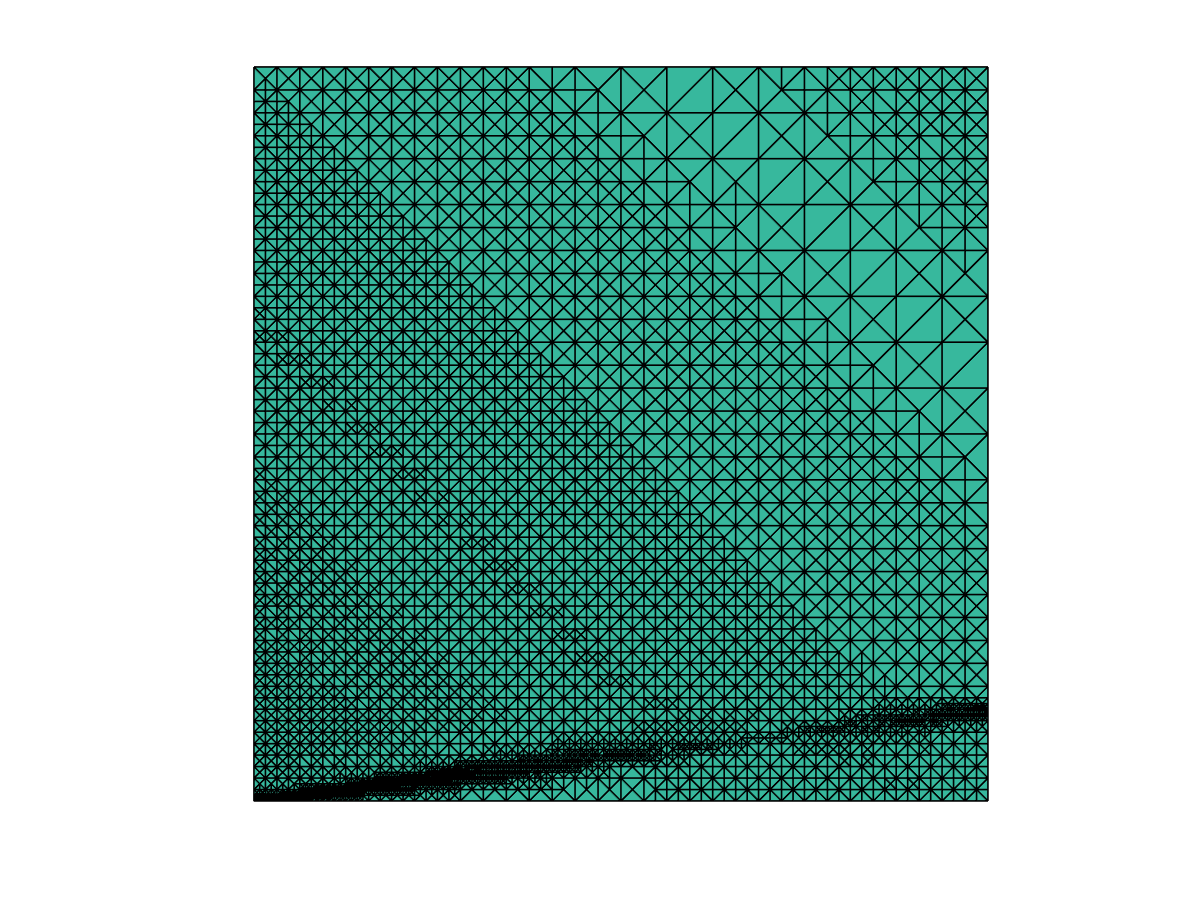}}
\hspace{0.01\linewidth}
\subfigure[convergence histories on adaptive refined meshes]{
\includegraphics[width=0.45\linewidth]{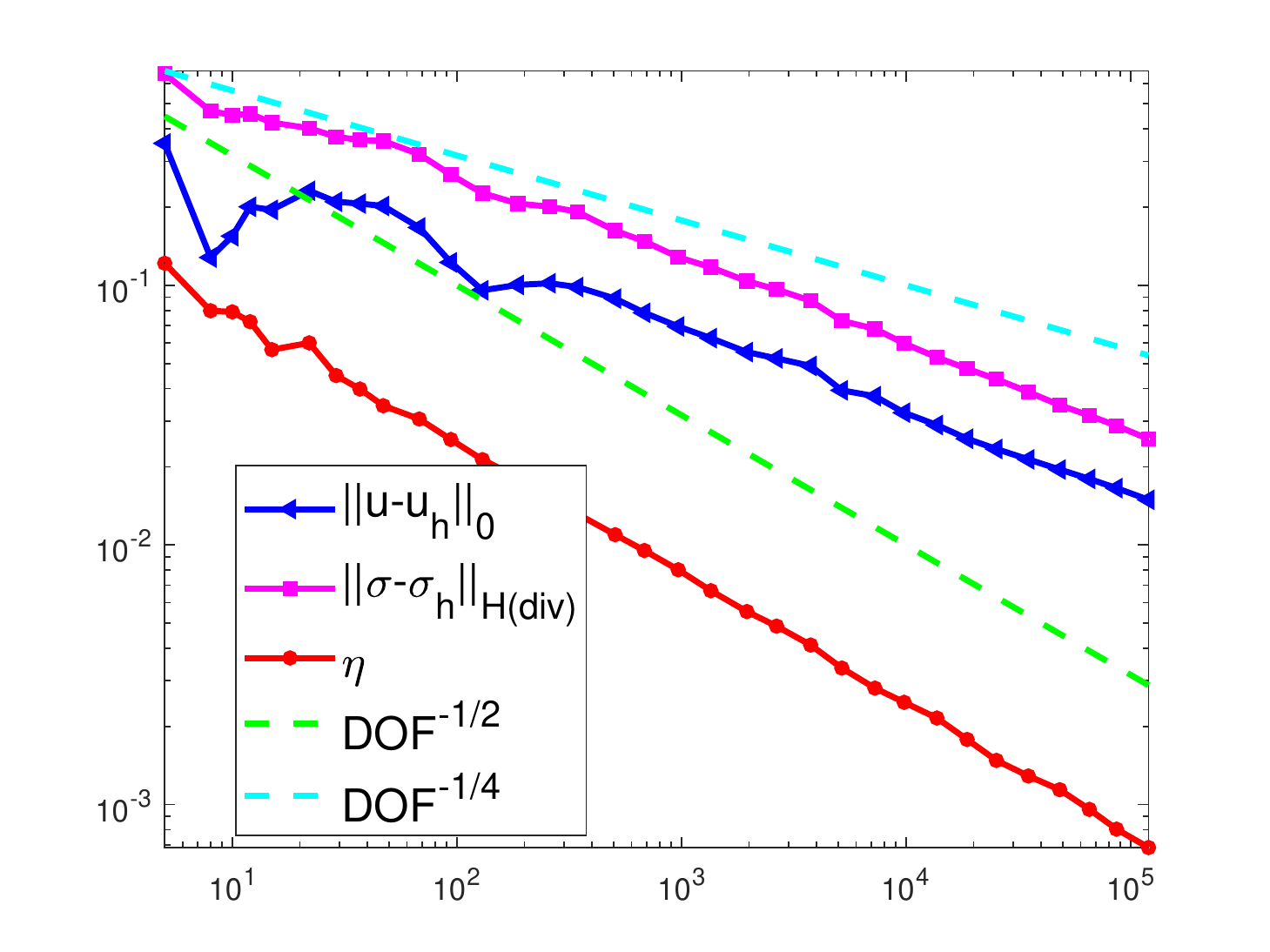}}
\caption{Piecewise smooth solution on a non-matching grid}
\label{adaptive_mesh_pws}
\end{figure}

\subsection{Curved transport example}
We consider an example similar to Example in 4.4.2 of \cite{Guermond:04}. Consider the problem on the half disk
$\O = \{(x,y) \colon x^2+y^2<1; y>0\}$. Let the inflow boundary be $\{-1<x<0; y=0\}$. Choose the advection field
$\bbeta = (\sin \theta, -\cos \theta)^T = (y/\sqrt{x^2+y^2}, -x/\sqrt{x^2+y^2})^T$, with $\theta$  the polar angle.
Let $\gamma=0$, $f=0$, and the inflow condition and the exact solution be
$$
g = \left\{ \begin{array}{lll}
1 & \mbox{if}  &-1<x<-0.5, \\[2mm]
0 & \mbox{if}  &-0.5<x<0,
\end{array} \right.
\mbox{and}\quad
u = \left\{ \begin{array}{lll}
1 & \mbox{if  }  x^2+y^2 > 0.25, \\[2mm]
0 & \mbox{otherwise}.
\end{array} \right.
$$

\begin{figure}[!htbp]
    \begin{minipage}[!hbp]{0.48\linewidth}
        \centering
        \includegraphics[width=0.99\textwidth,angle=0]{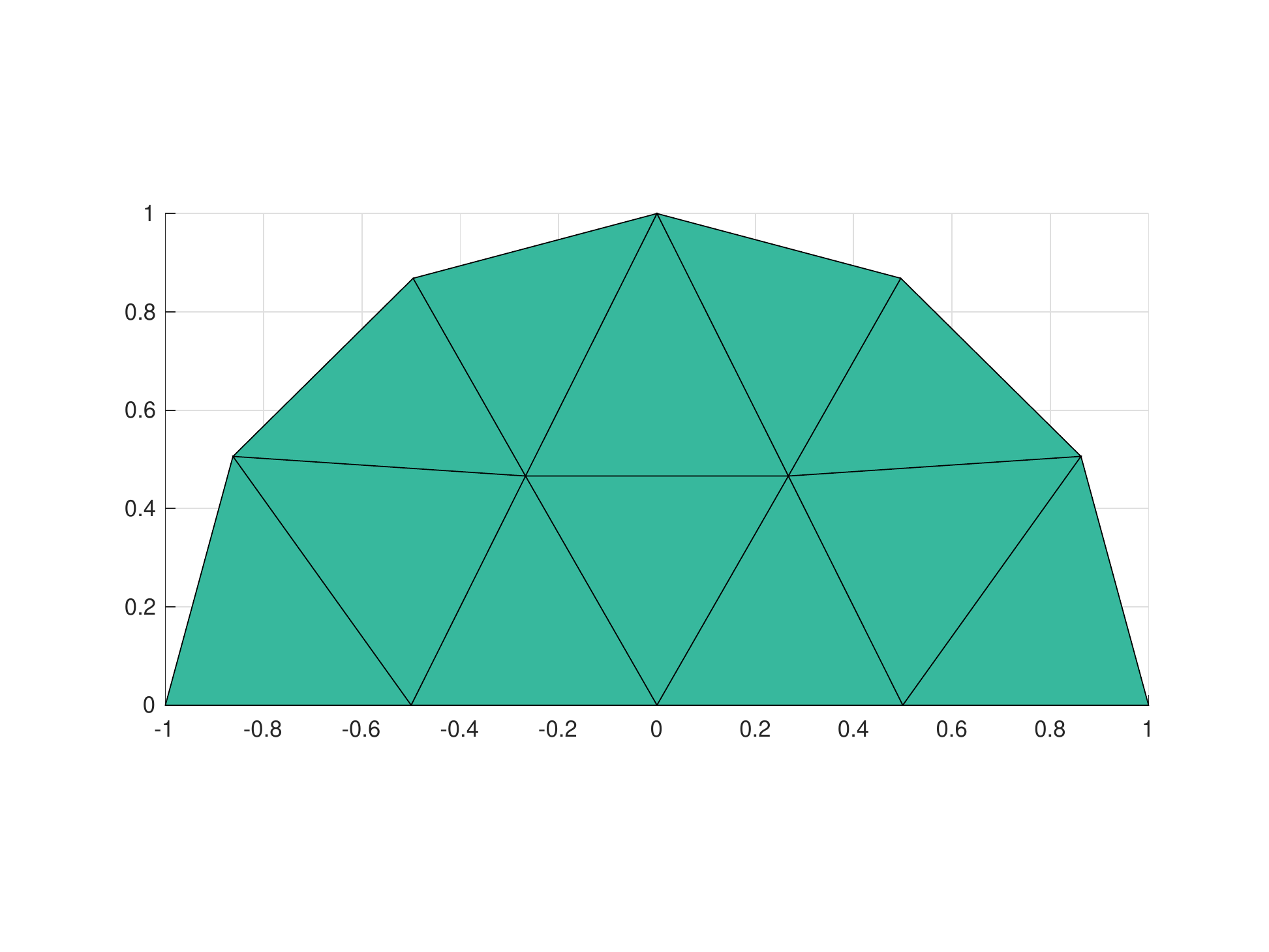}
        \end{minipage}%
        \quad
    \begin{minipage}[!htbp]{0.48\linewidth}
        \centering
        \includegraphics[width=0.99\textwidth,angle=0]{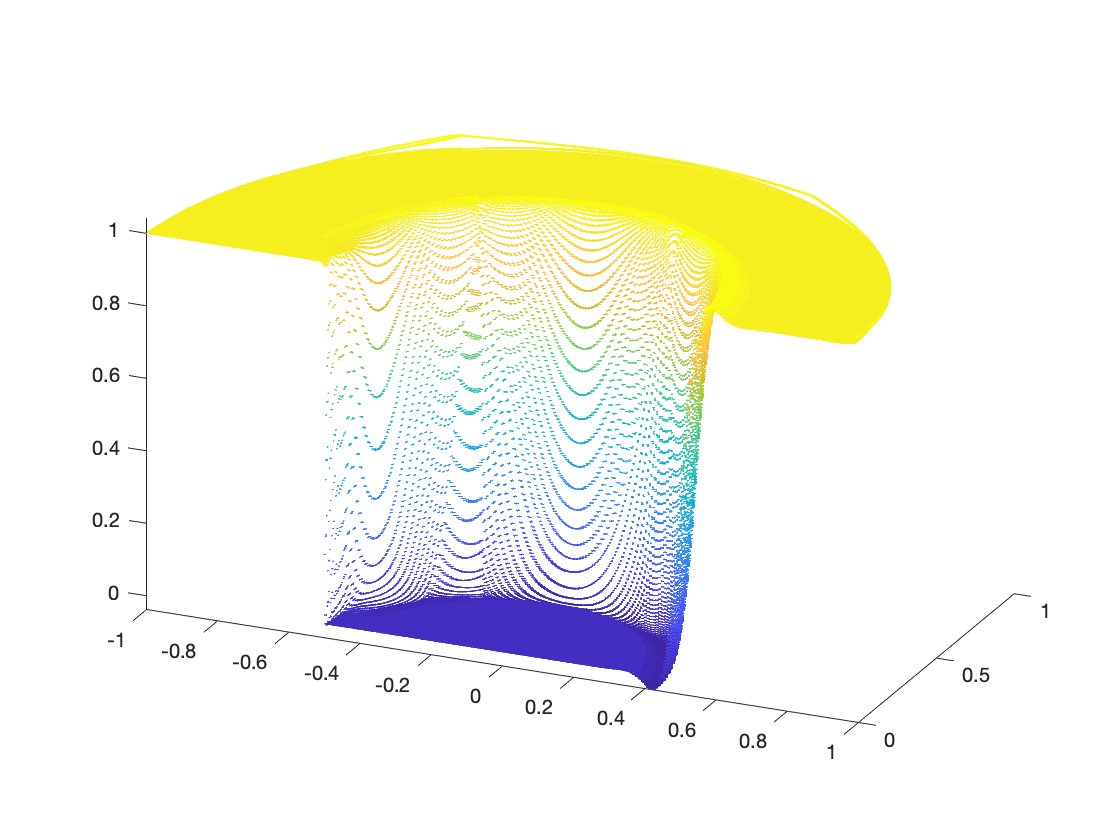}
        \end{minipage}
    \caption{Curved transport problem: initial mesh(left), numerical solution (LSFEM 1-1) on an almost uniform mesh (right)}%
        \label{curved_initialmesh}
\end{figure}

We use an initial mesh as shown on the left of Fig. \ref{curved_initialmesh}. We choose the bottom central node to be $(0,0)$ and the node left of it to be $(-0.5,0)$. So the inflow boundary mesh is alighed with the inflow boundary discontinuity. Since the advection field is curved and so is the discontinuity, the mesh will never be aligned with the discontinuity even after refinements. Since the boundary is a half circle, when we refine the mesh, we take an extra step to map those boundary nodes to the right positions on the circle.

Since $\gamma =0$ in this example, so we only use the LSFEM1-1 to compute it.

We show the numerical recovered solution $u_h$ computed by LSFEM1-1  on a mesh after 8 uniform refinements of the initial mesh on the right of Fig.  \ref{curved_initialmesh}.
Small overshootings can be observed near the discontinuity.
Along the radius, the solution is essentially one dimensional,
we project the graph of the solution onto the radius, see the left of Fig. \ref{curved_solutionuniform_projected}.
We do see the under and overshooting. The maximum and minimum values of numerical solution $u_h$ are $1.0404$ and $-0.0383$, respectively.

With uniform refinements,
the convergence rate of the error in the least-squares norm is about $0.75$
and the rates of $\|u-u_h\|_0$ and $\|\bsigma-\bsigma_h\|_{H(div)}$ are about $0.28$,
see the right of Fig. \ref{curved_solutionuniform_projected}.
Since the mesh is not aligned with the discontinuities, the convergence order of
the least-sqaures energy norm is smaller than $1$.

\begin{figure}[!htb]
    \begin{minipage}[!hbp]{0.48\linewidth}
        \centering
        \includegraphics[width=0.99\textwidth,angle=0]{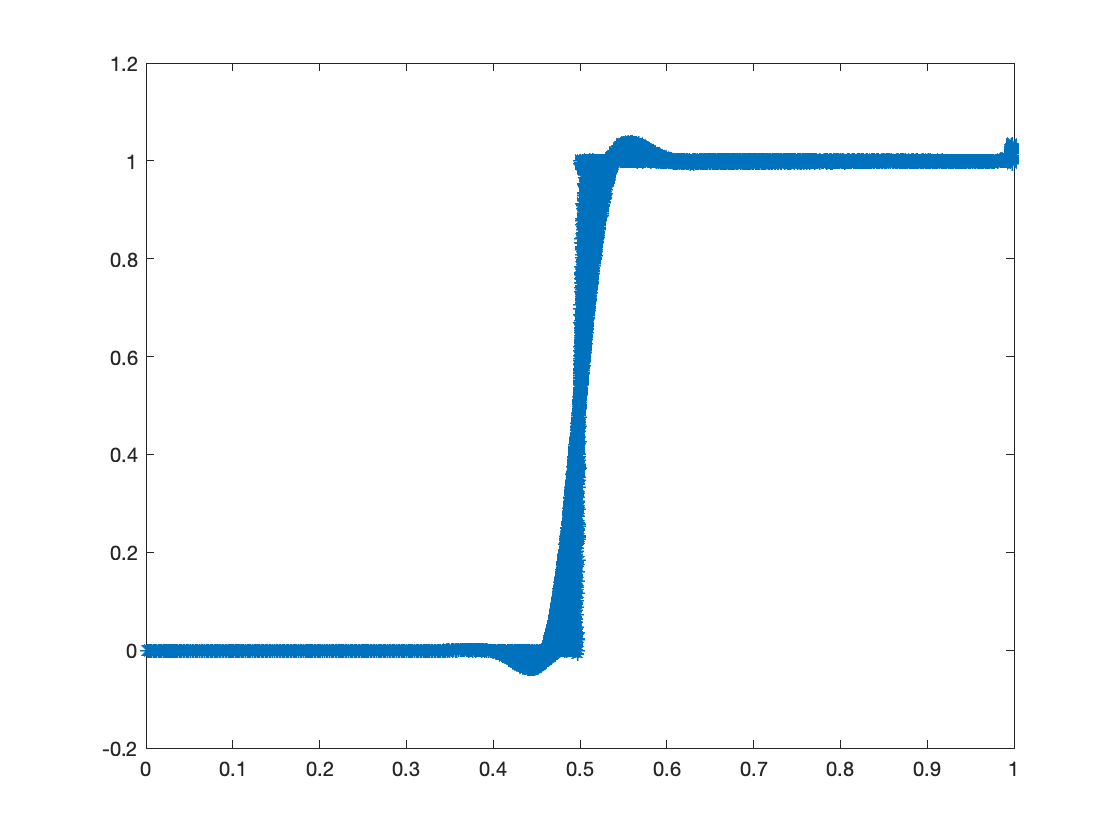}
        \end{minipage}%
        \quad
    \begin{minipage}[!htbp]{0.48\linewidth}
        \centering
        \includegraphics[width=0.99\textwidth,angle=0]{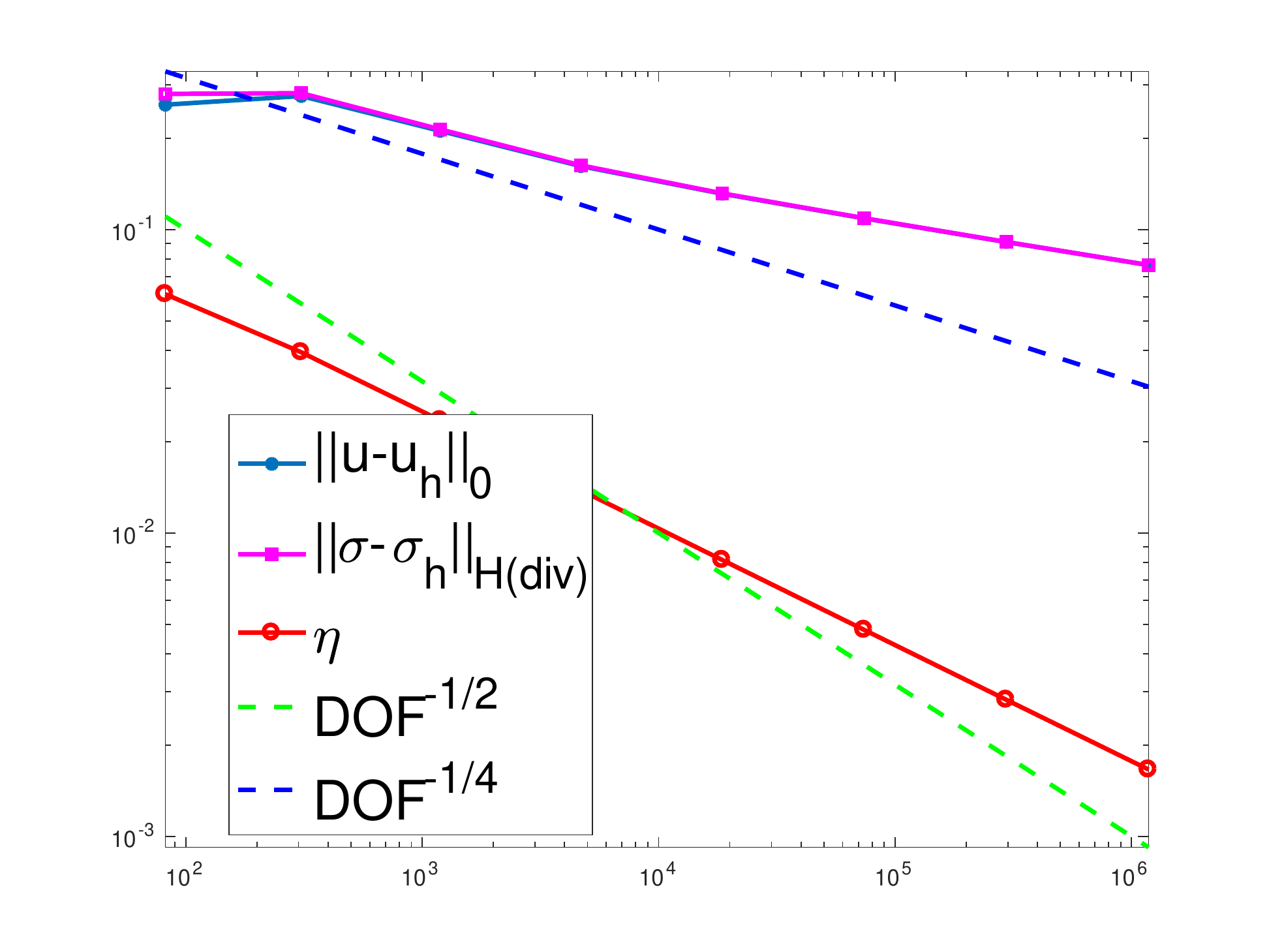}
        \end{minipage}
    \caption{Curved transport problem: projected numerical solutions on an almost uniform mesh (left), convergence behaviors on uniformly refined meshes (right)}%
        \label{curved_solutionuniform_projected}
\end{figure}

On the left of Fig. \ref{curved_adaptivemesh}, we show the adaptive mesh generated by LSFEM1-1 after several iterations. Along the discontinuity, we see many refinements. Also, almost uniform refinements can be found in the half ring where $u = 1$. The reason is that even $u$ is a constant $1$, the flux $\bsigma = \bbeta$ is not a constant vector and has approximation errors. On the other hand, in the region where $u=0$, the flux is also a zero vector and can be exactly computed. So no refinement is needed in the inner half ring.

On the right of  Fig. \ref{curved_adaptivemesh}, we show the convergence history of the adaptive method. With the adaptive least-squares finite element method, the convergence order of the error in the least-squares norm is about $1$ and is optimal, and the rates of $\|u-u_h\|_0$ and $\|\bsigma-\bsigma_h\|_{H(\divvr)}$ are about $0.5$.

\begin{figure}[!htb]
    \begin{minipage}[!hbp]{0.48\linewidth}
        \centering
        \includegraphics[width=0.99\textwidth,angle=0]{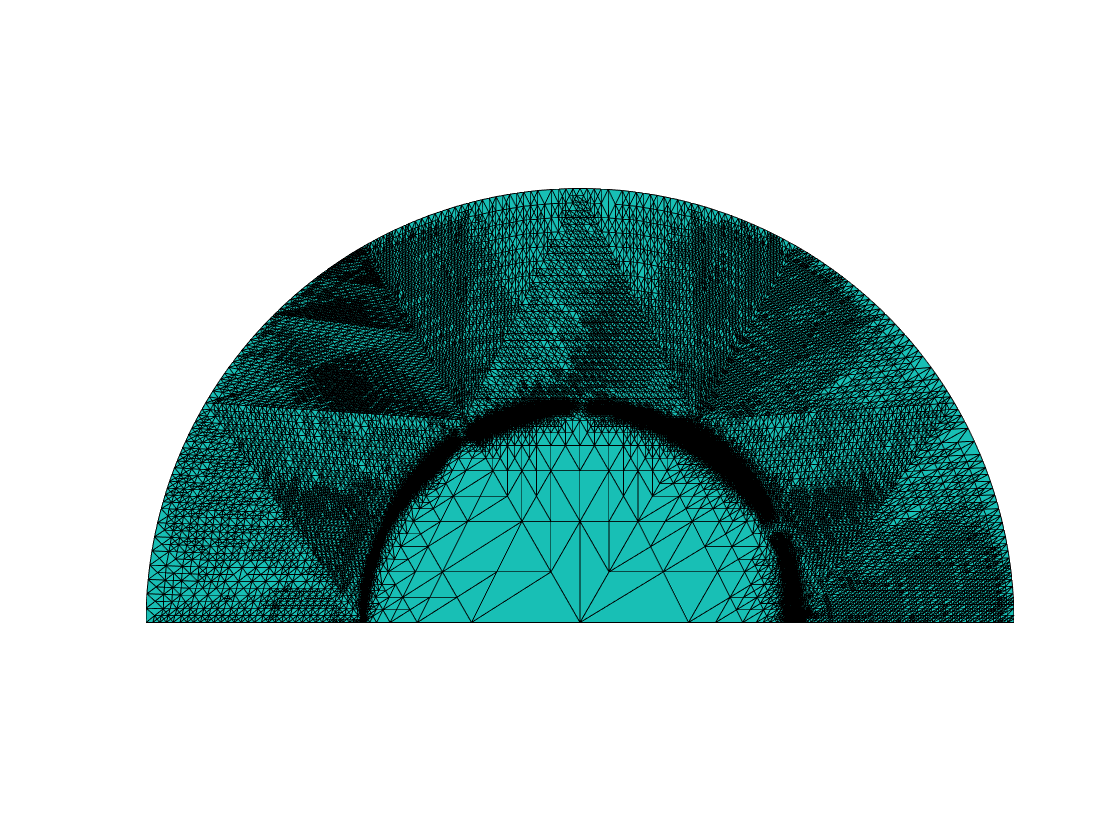}
        \end{minipage}%
        \quad
    \begin{minipage}[!htbp]{0.48\linewidth}
        \centering
        \includegraphics[width=0.99\textwidth,angle=0]{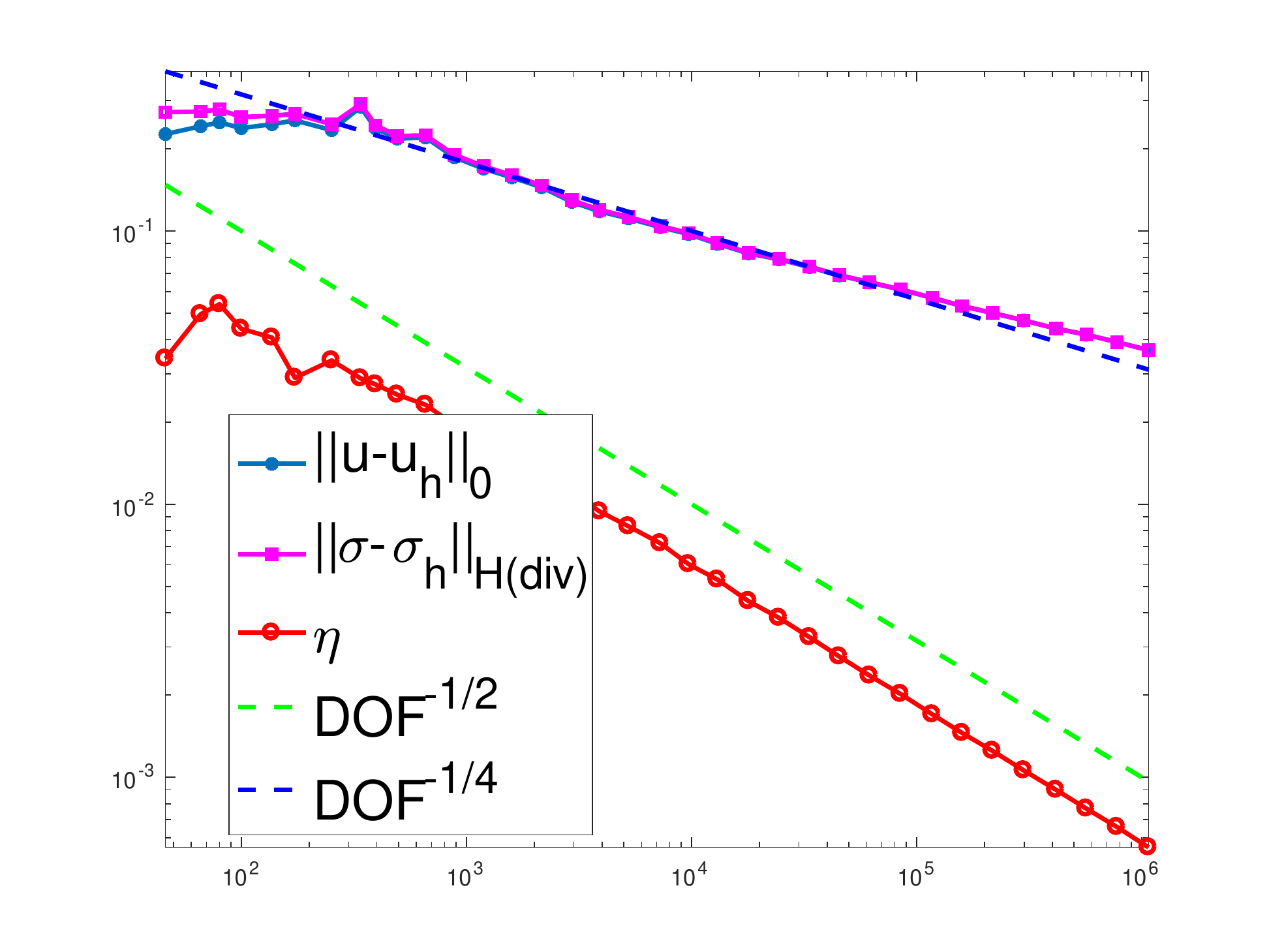}
        \end{minipage}
    \caption{Curved transport problem: adaptively refined meshes after several iterations (left), convergence behaviors on adaptively refined meshes (right)}%
    \label{curved_adaptivemesh}
\end{figure}

On the left of Fig. \ref{curved_os}, we show the reduction of overshooting values obtained by LSFEM1-1. After the initial stages, the overshooting values are decreasing with refinements along the discontinuities. On the right of Fig. \ref{curved_os}, the projected solution on the radius is shown on the final mesh. We can see that the overshooting is very small compared with the uniform refinements. The Gibbs phenomena is not observed.

\begin{figure}[!ht]
    \begin{minipage}[!hbp]{0.48\linewidth}
        \centering
        \includegraphics[width=0.99\textwidth,angle=0]{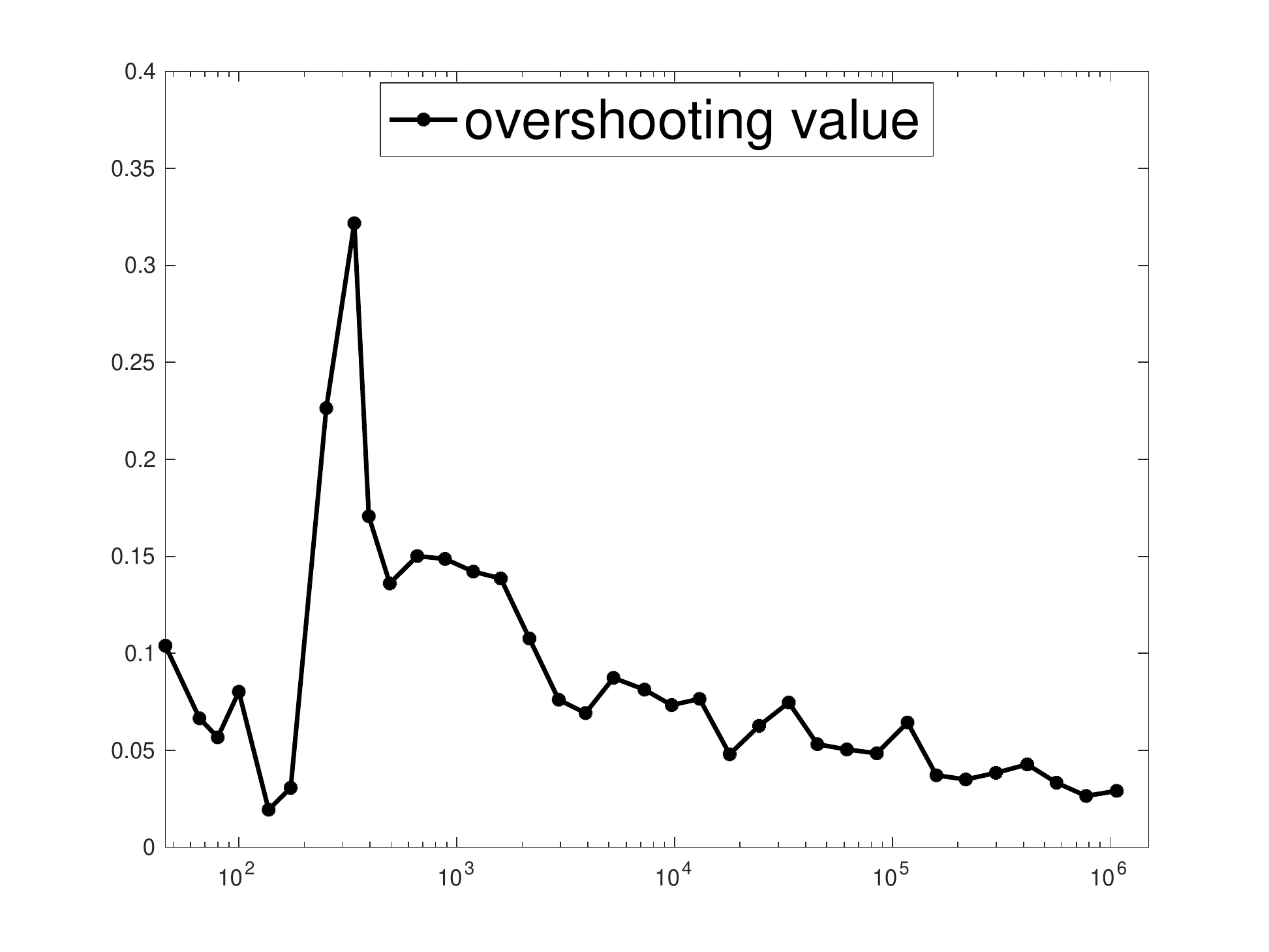}
        \end{minipage}%
        \quad
    \begin{minipage}[!htbp]{0.48\linewidth}
        \centering
        \includegraphics[width=0.99\textwidth,angle=0]{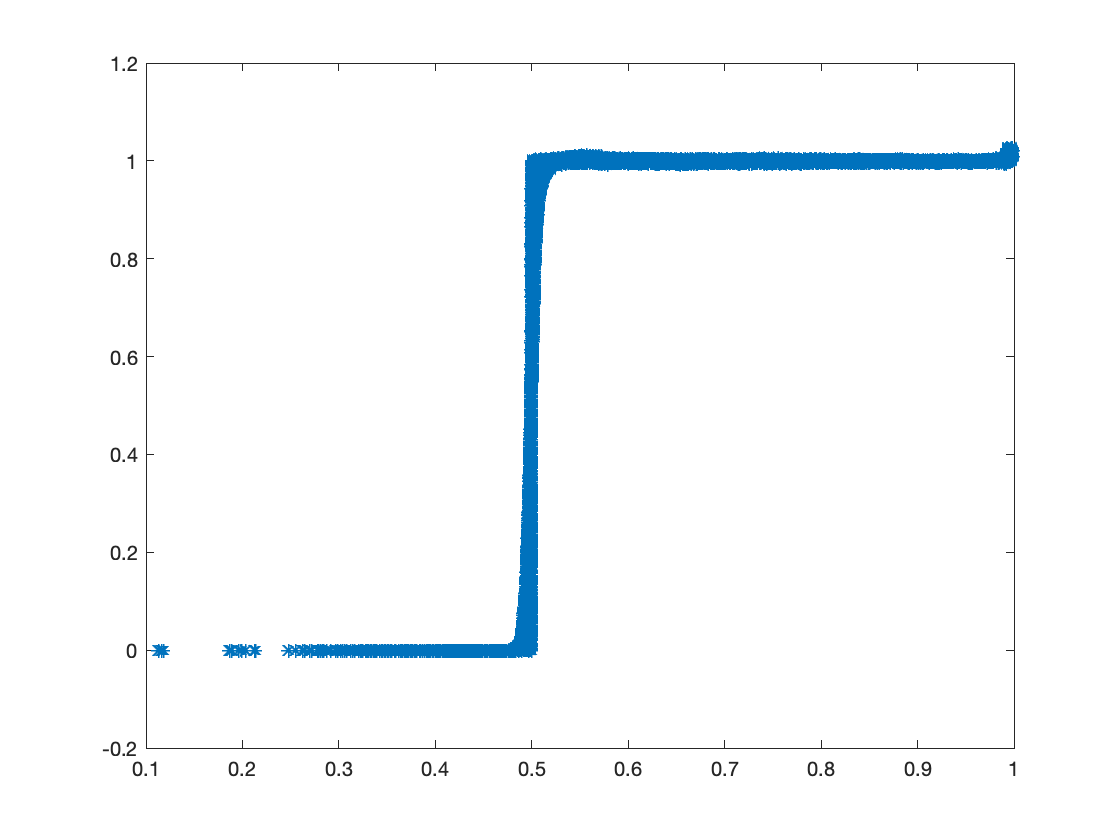}
        \end{minipage}
    \caption{Curved transport problem: reduction of overshootings by adaptive LSFEM (left), projection solution on an adaptively refined mesh (right)}%
    \label{curved_os}
\end{figure}

\subsection{A smooth example with a sharp transient layer}
Consider the following problem:  $\bbeta = (y+1,-x)^T/\sqrt{x^2+(y+1)^2}$, $\O = (0,1)^2$,  $\gamma=0.1$, and $f=0$. The inflow boundary is $\{x=0, y\in (0,1)\} \cup \{x\in (0,1), y=1\}$, i.e., the west and north boundaries of the domain. Choose $g$ such that the exact solution $u$ is
$$
u = \dfrac{1}{4}\exp\left(\gamma r\arcsin \left(\dfrac{y+1}{r}\right)\right)
\arctan \left(\dfrac{r-1.5}{\epsilon}\right), \quad \mbox{with}\quad r = \sqrt{x^2+(y+1)^2}.
$$

When $\epsilon = 0.01$, the layer can be fully resolved. 
When $\epsilon = 10^{-10}$, the layer is never fully resolved in our experiments and can be viewed as discontinuous.

\begin{figure}[!ht]
\centering 
\subfigure[$\epsilon=0.01$]{ 
\includegraphics[width=0.45\linewidth]{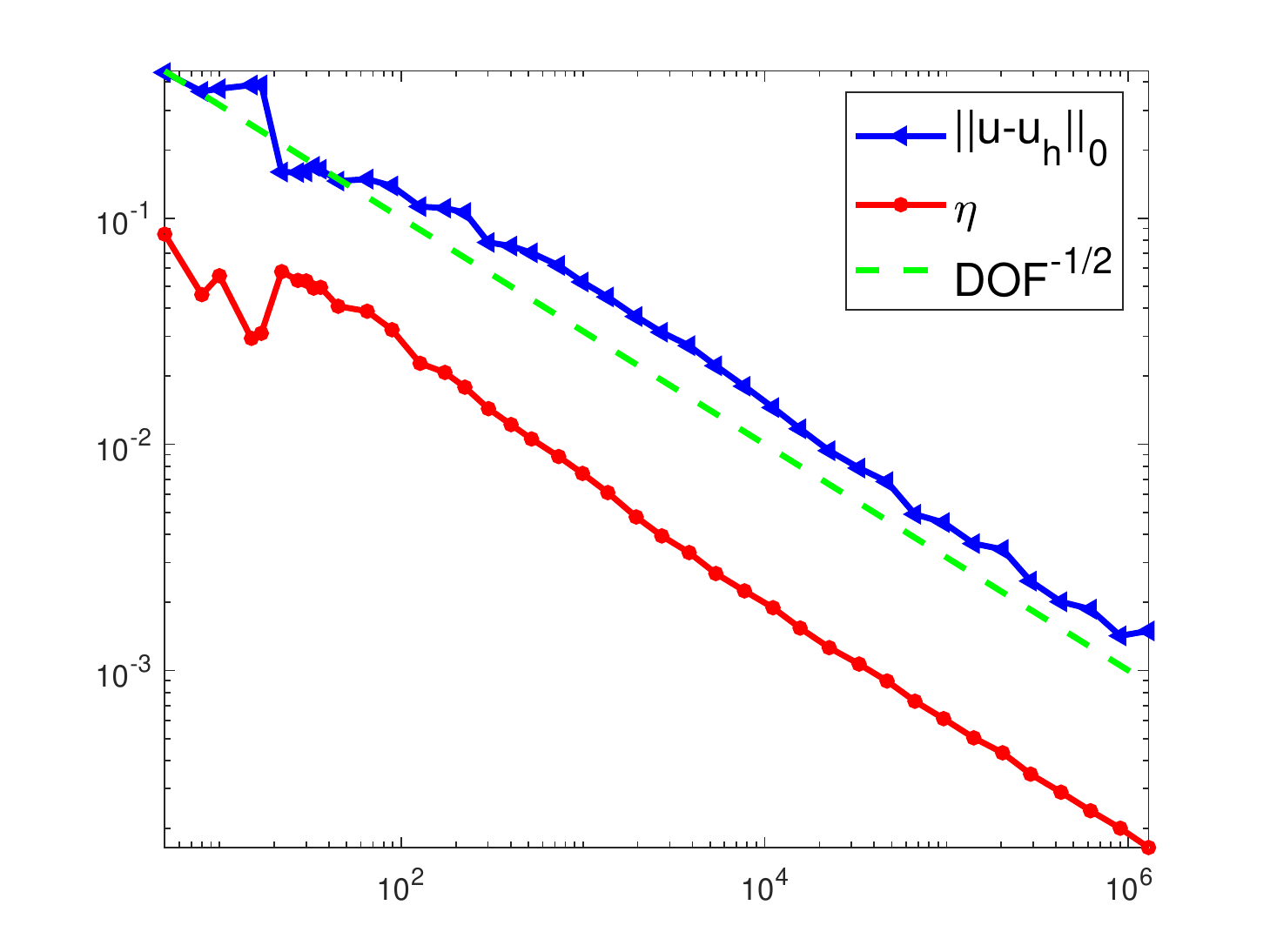}}
\hspace{0.01\linewidth}
\subfigure[$\epsilon=10^{-10}$]{ 
\includegraphics[width=0.45\linewidth]{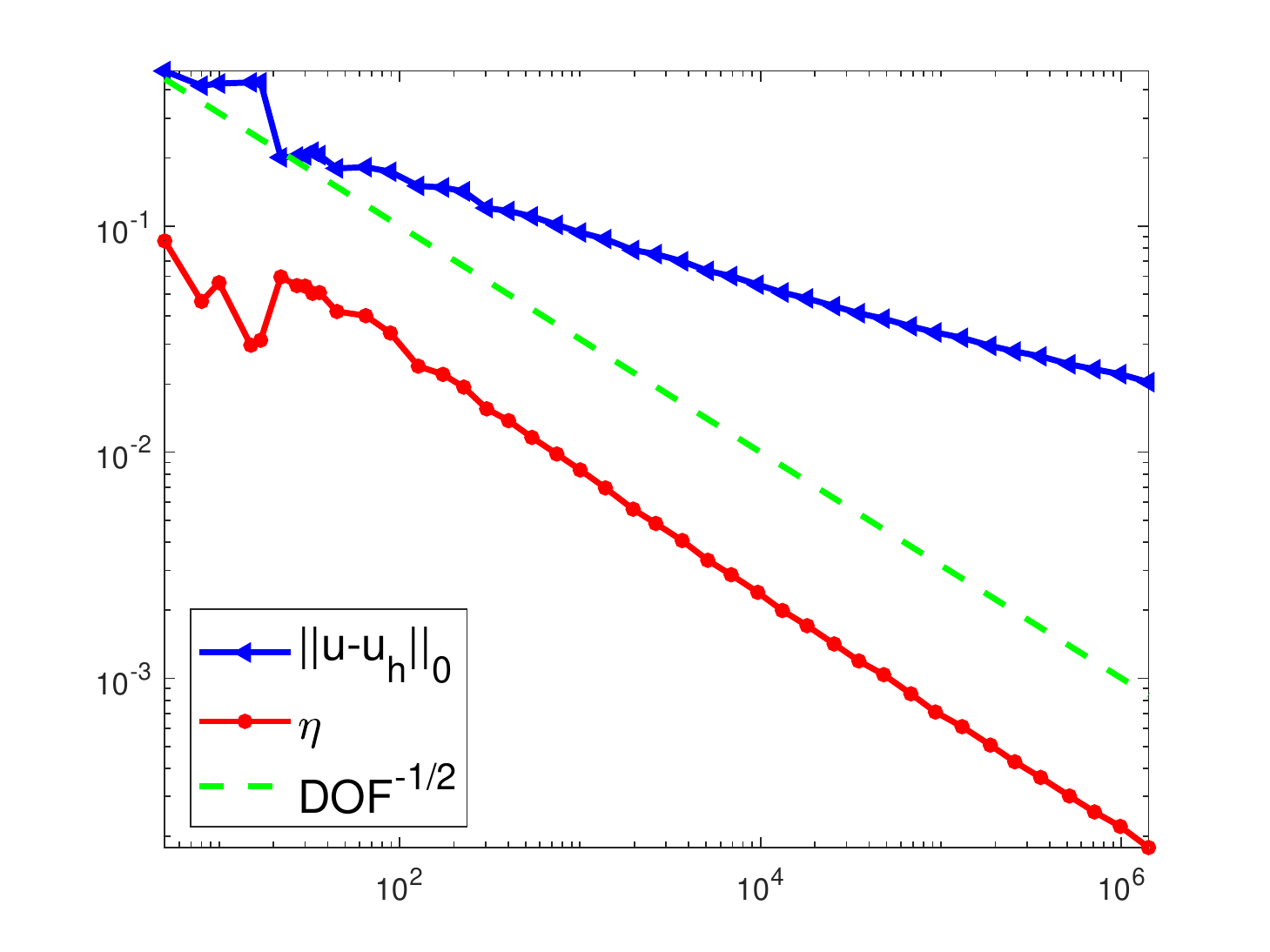}}
\caption{Transient layer problem:  adaptive convergence behaviors}
\label{adaptive_error_1e2}
\end{figure}

\begin{figure}[!htb]
\centering 
\subfigure[adaptive  meshe after several iterations]{ 
\includegraphics[width=0.45\linewidth]{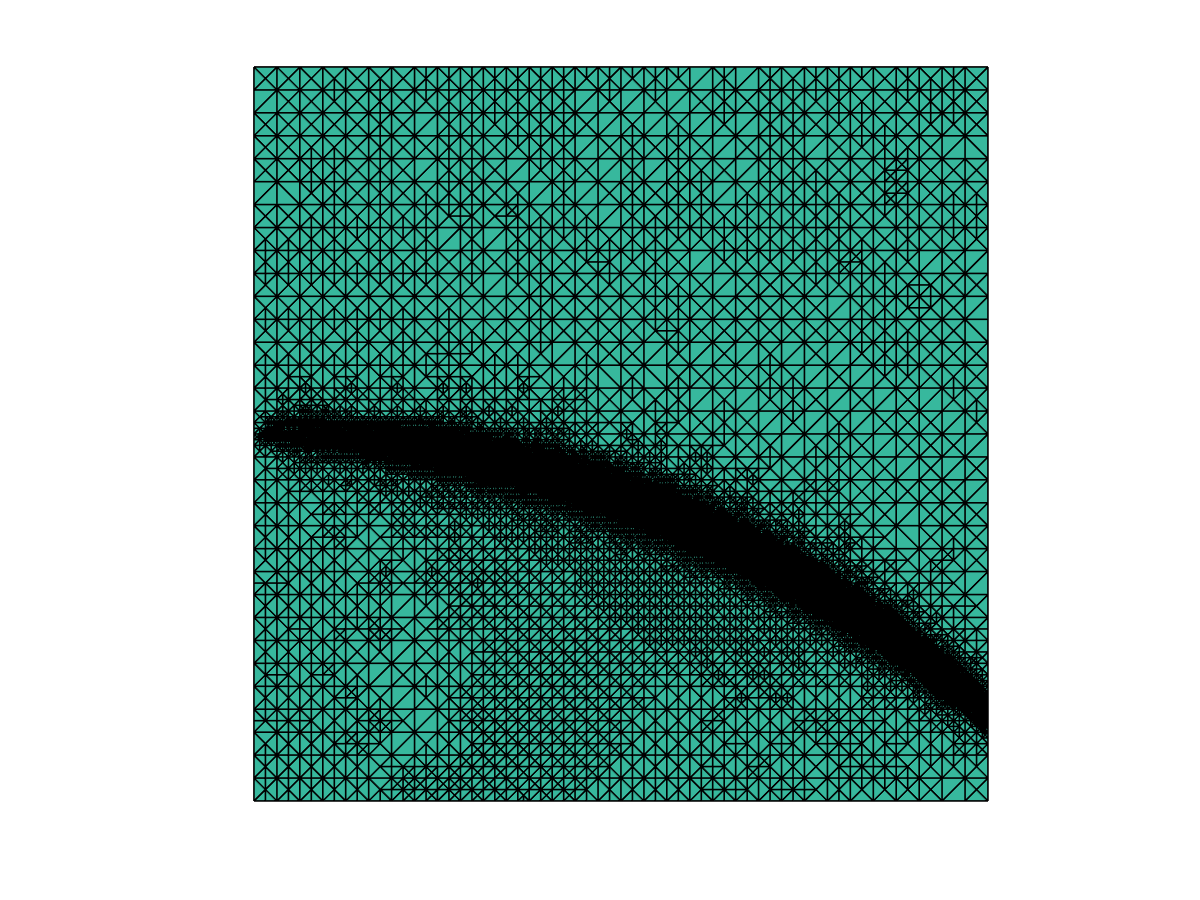}}
\hspace{0.01\linewidth}
\subfigure[contour of the solution]{
\includegraphics[width=0.45\linewidth]{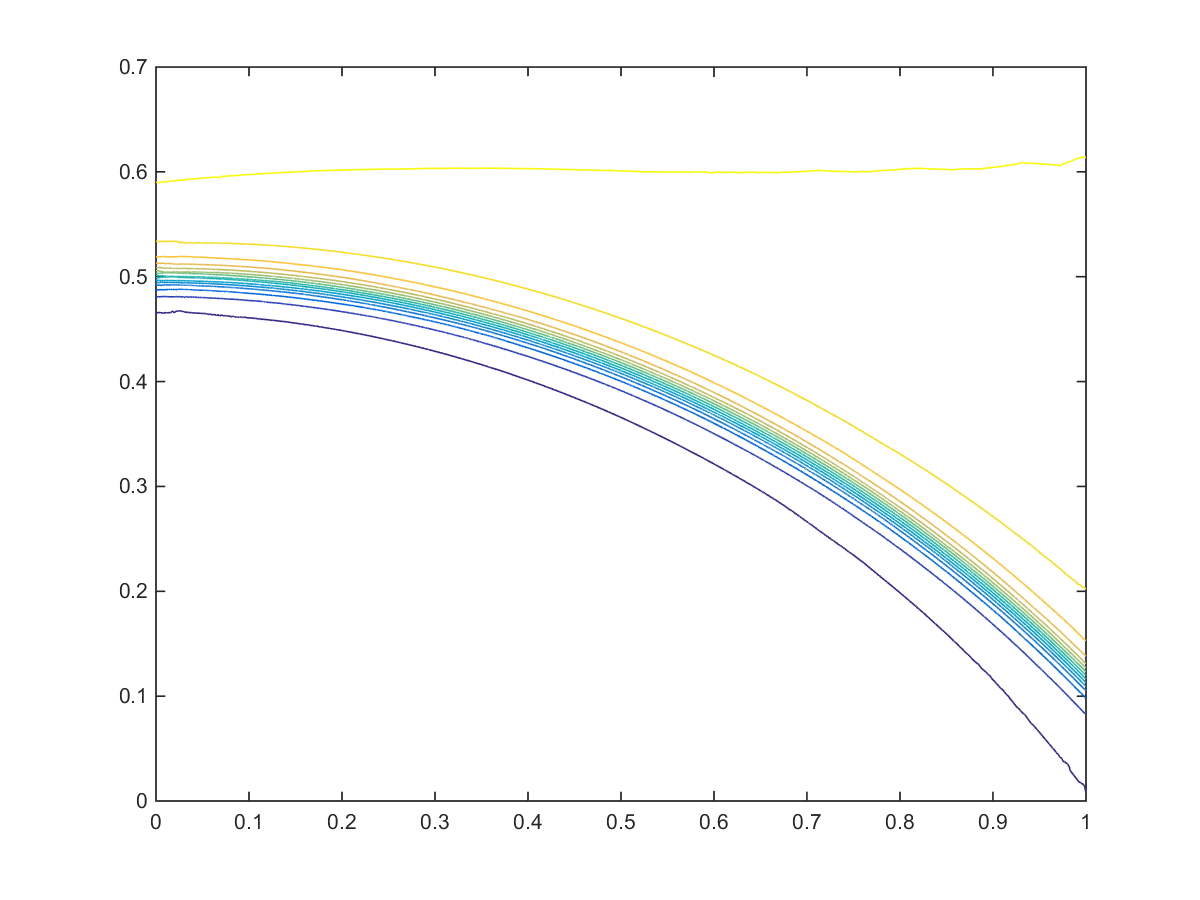}}
\caption{Transient layer problem $\epsilon=0.01$}
\label{adaptive_mesh_1e2}
\end{figure}

For $\epsilon = 0.01$, we show the numerical results in  Figs. \ref{adaptive_mesh_1e2} and \ref{adaptive_error_1e2}. The behaviors of the methods are very similar to the global continuous solution case. For $\epsilon = 10^{-10}$, we show the numerical results in  Figs. \ref{adaptive_error_1e2} and \ref{adaptive_mesh_1e10}. The behaviors of the methods are very similar to the piecewise smooth solution example with a non-matching grid, the example 5.6. The order of convergence of $\|u-u_h\|_0$ is about $0.12$. The contour of the solution on the right of Fig. \ref{adaptive_mesh_1e10} shows that the overshooting is neglectable when the mesh is fine enough.

\begin{figure}[!ht]
\centering 
\subfigure[adaptive mesh after several iterations]{ 
\includegraphics[width=0.45\linewidth]{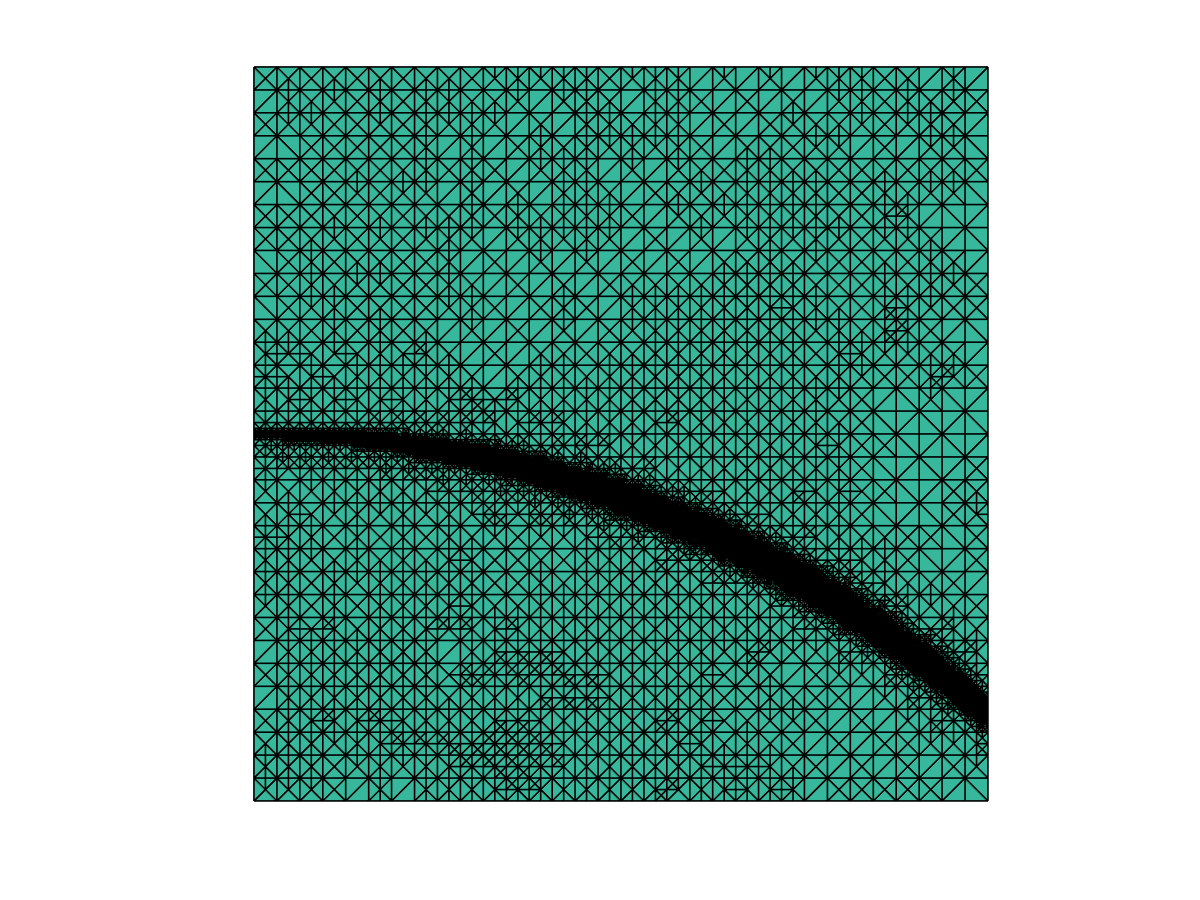}}
\hspace{0.01\linewidth}
\subfigure[contour of solution]{
\includegraphics[width=0.45\linewidth]{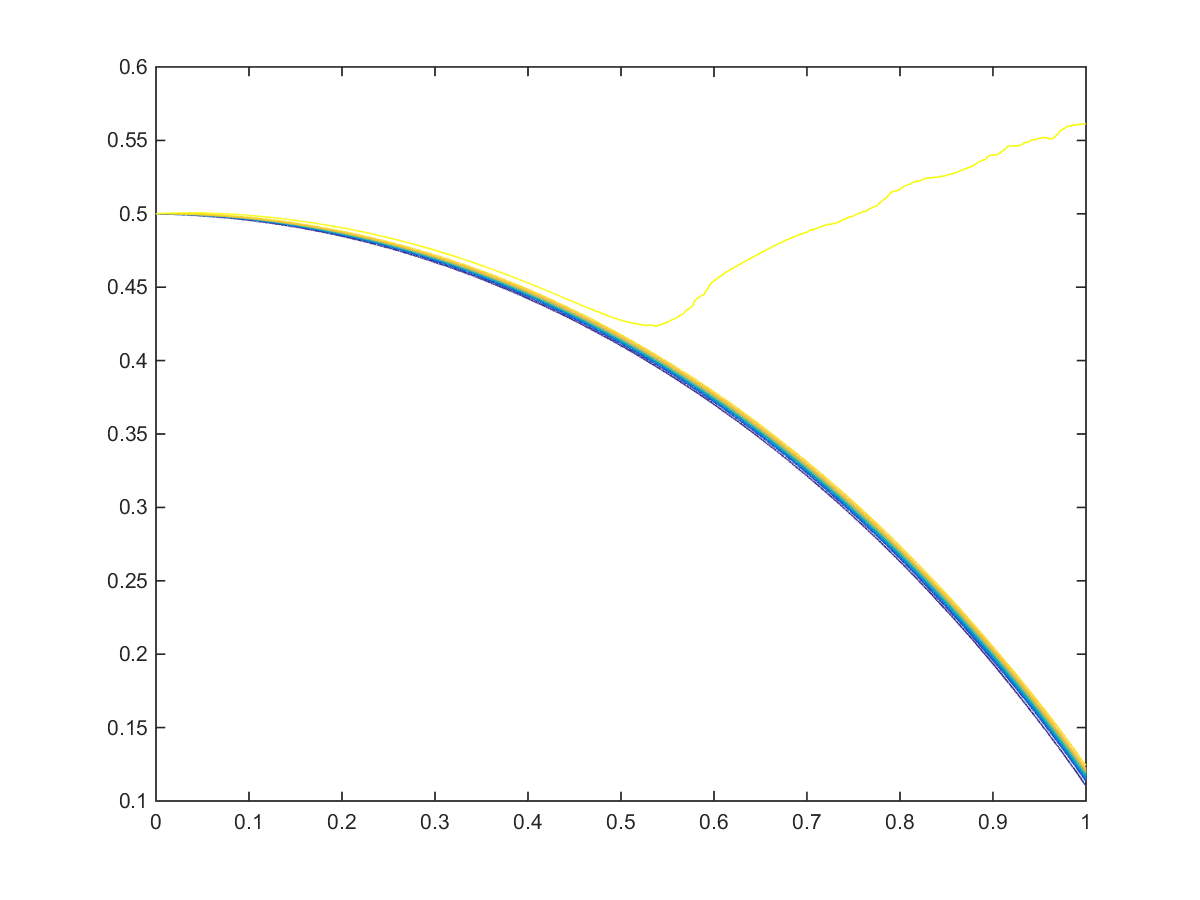}}
\caption{Transient layer problem $\epsilon=10^{-10}$}
\label{adaptive_mesh_1e10}
\end{figure}

\begin{rem}
The numerical tests here show that the convergence rate of the flux-only LSFEMs is the same as the original flux-solution LSFEMs that we suggested in \cite{LZ:18}. 

As mentioned in \cite{LZ:18}, the common rumor that of the least-squares method is it tends to have a strong smearing effect which is actually often the effect of Galerkin least-squares or stabilized methods where some least-squares terms are added to variational problems, see for example \cite{HFH:88}. We do not observe smearing effect for the bona fide least-squares methods developed in this paper.
\end{rem}

\section{Concluding Remarks}
\setcounter{equation}{0}
In this paper, two flux-only least-squares finite element methods (LSFEM) for the linear hyperbolic transport problem are developed. We first reformulate the linear transport equation into a flux-solution system, then eliminate the solution from the system and get two flux-only formulations, and develop corresponding LSFEMs. The solution then is recovered by simple post-processing methods using its relation with the flux. These two versions of flux-only LSFEMs use less DOFs than the method we developed in \cite{LZ:18}.
	
In this paper, we only discuss the versions of LSFEMs with boundary condition strongly enforced. We can also weakly enforce the condition as we did in \cite{LZ:18}.
		
Similar to the LSFEM developed in \cite{LZ:18}, both flux-only LSFEMs can handle discontinuous solutions better than the traditional continuous polynomial approximations. We show the existence, uniqueness, a priori and a posteriori error estimates of the proposed methods.  With adaptive mesh refinements driven by the least-squares a posteriori error estimators, the solution can be accurately approximated even when the mesh is not aligned with discontinuity. The overshooting phenomenon is very mild if a piecewise constant reconstruction of the solution is used. Extensive numerical tests are done to show the effectiveness of the methods developed in the paper.

\section*{Acknowledgement}
S. Zhang is supported in part by Hong Kong Research Grants Council under the GRF Grant Project No. 11305319, CityU and a China Sichuan Provincial Science and Technology Research Grant 2018JY0187 via Chengdu Research Institute of City University of Hong Kong.

\bibliographystyle{siam}
\bibliography{ls_transport}

\end{document}

%% file: defs.tex
\newcommand{\BX}{{\bf X}}
\newcommand{\cv}{{\cal V}}
\newcommand{\cW}{{\cal W}}
\newcommand{\co}{{\cal O}}

\renewcommand{\theequation}{\thesection.\arabic{equation}}
\def\@eqnnum{{\reset@font\rm (\theequation)}}




\def\a{\alpha}
\def\b{\beta}
\def\d{\delta}\def\D{\Delta}
\def\e{\epsilon}
\def\g{\gamma}\def\G{\Gamma}
\def\k{\kappa}
\def\lam{\lambda}\def\Lam{\Lambda}
\renewcommand\o{\omega}\renewcommand\O{\Omega}
\def\s{\sigma}\def\S{\Sigma}
\renewcommand\t{\theta}\def\vt{\vartheta}
\newcommand{\vphi}{\varphi}
\def\z{\zeta}

\newcommand{\tsigma}{\tilde{\s}}
\newcommand{\tbsigma}{\tilde{\bsigma}}
\def\te{\tilde{\e}}
\def\tu{\tilde{u}}

\newcommand{\bchi}{\mbox{\boldmath$\chi$}}
\newcommand{\bdelta}{\mbox{\boldmath$\delta$}}
\newcommand{\bepsilon}{\mbox{\boldmath$\epsilon$}}
\newcommand{\bfeta}{\mbox{\boldmath$\eta$}}
\newcommand{\bgamma}{\mbox{\boldmath$\gamma$}}
\newcommand{\bomega}{\mbox{\boldmath$\omega$}}
\newcommand{\bvphi}{\mbox{\boldmath$\varphi$}}
\newcommand{\bphi}{\mbox{\boldmath$\phi$}}
\newcommand{\bPhi}{\mbox{\boldmath$\Phi$}}
\newcommand{\bpsi}{\mbox{\boldmath$\psi$}}
\newcommand{\bPsi}{\mbox{\boldmath$\Psi$}}
\newcommand{\bsigma}{\mbox{\boldmath$\sigma$}}
\newcommand{\btau}{\mbox{\boldmath$\tau$}}
\newcommand{\bxi}{\mbox{\boldmath$\xi$}}
\newcommand{\brho}{\mbox{\boldmath$\rho$}}
\newcommand{\bbeta}{\mbox{\boldmath$\beta$}}
\newcommand{\bzeta}{\mbox{\boldmath$\zeta$}}

\def\bk{\boldsymbol{\kappa}}
\def\bmu{\boldsymbol\mu}
\def\bxi{\boldsymbol{\xi}}
\def\bz{\boldsymbol{\zeta}}

\def\ba{{\bf a}}
\def\bb{{\bf b}}
\def\bc{{\bf c}}
\def\be{{\bf e}}
\def\bff{{\bf f}}
\def\bg{{\bf g}}
\def\bn{{\bf n}}
\def\bp{{\bf p}}
\def\bq{{\bf q}}
\def\bs{{\bf s}}
\def\bt{{\bf t}}
\def\bu{{\bf u}}
\def\bv{{\bf v}}
\def\bw{{\bf w}}
\def\bx{{\bf x}}
\def\by{{\bf y}}
\def\bzz{{\bf z}}

\def\bD{{\bf D}}
\def\bE{{\bf E}}
\def\bF{{\bf F}}
\def\bH{{\bf H}}
\def\bJ{{\bf J}}
\def\bV{{\bf V}}
\def\bU{{\bf U}}
\def\bW{{\bf W}}
\def\bX{{\bf X}}
\def\bY{{\bf Y}}

\def\cA{{\cal A}}
\def\cC{{\cal C}}
\def\cD{{\cal D}}
\def\cE{{\cal E}}
\def\cF{{\cal F}}
\def\cG{{\cal G}}
\def\cI{{\cal I}}
\def\cJ{{\cal J}}
\def\cK{{\cal K}}
\def\cL{{\cal L}}
\def\cO{{\cal O}}
\def\cP{{\cal P}}
\def\cQ{{\cal Q}}
\def\cR{{\cal R}}
\def\cS{{\cal \Sigma}}
\def\cT{{\cal T}}
\def\cU{{\cal U}}
\def\cV{{\cal V}}

\def\scT{{_\cT}}
\def\sD{{_D}}
\def\sE{{_E}}
\def\sF{{_F}}
\def\sFz{{_{F_z}}}
\def\sK{{_K}}
\def\sI{{_I}}
\def\sb{{_b}}
\def\sN{{_N}}

\def\curl{{{\bf curl} \ }}
\def\rot{{\mbox{rot}\ }}
\def\BPI{{\bf \Pi}}

\def\cth{\cT_h}
\def\ctH{\cT_H}

\def\tJ{\tilde{\J}}

\def\hK{\widehat{K}}
\def\hx{\widehat{x}}
\def\hy{\widehat{y}}
\def\bhv{\widehat{\bv}}

\def\l{\ell}
\def\bl{\boldsymbol{\ell}}
\def\col{\colon}
\def\f12{\frac12}
\def\dfrac{\displaystyle\frac}
\def\dint{\displaystyle\int}
\def\nab{\nabla}
\def\p{\partial}
\def\sm{\setminus}
\def\dsum{\displaystyle\sum}
\newcommand{\pp}[2]{\frac{\partial {#1}}{\partial {#2}}}
\def\bzero{{\bf 0}}

\def\divv{\nab\cdot}
\def\divx{\nab_x\cdot}
\def\divtx{\nab_{t,x}\cdot}
\def\nabx{\nab_x}

\newcommand{\grad}{\nabla}
\newcommand{\curlt}{{\nabla \times}}
\newcommand{\gperp}{\nabla^{\perp}}
\newcommand{\gradt}{\nabla\cdot}

\def\forallqq{\quad\forall\,}
\def\aph{A^{1/2}}
\def\amh{A^{-1/2}}

\def\osc{{\rm osc \, }}

\def\Im{{\rm Im}}
\newcommand{\tr}{{\rm tr}}
\def\divvr{{\rm div}}
\def\curllr{{\rm curl}}
\def\curll{{\rm curl}}
\def\curl{{\bf curl}}
\newcommand{\bgrad}{{\bf grad}}
\newcommand\diam{\mathrm{diam\,}}
\renewcommand\Im{\mathrm{Im\,}}
\def\Span{\mbox{Span}}
\def\supp{\mbox{supp\,}}
\newcommand{\trace}{{\rm trace}}

\newcommand{\tri}{|\!|\!|}
\newcommand{\ljump}{\lbrack\!\lbrack}
\newcommand{\rjump}{\rbrack\!\rbrack}
\newcommand{\bdm}{\begin{displaymath}}
\newcommand{\edm}{\end{displaymath}}
\newcommand{\beq}{\begin{equation}}
\newcommand{\eeq}{\end{equation}}
\newcommand{\beqa}{\begin{eqnarray}}
\newcommand{\eeqa}{\end{eqnarray}}
\newcommand{\beqas}{\begin{eqnarray*}}
\newcommand{\eeqas}{\end{eqnarray*}}
\newcommand{\ul}{\underline}
\newcommand{\wh}{\widehat}
\newcommand{\la}{\langle}
\newcommand{\ra}{\rangle}

\newcommand{\Lt}{L^2(\Omega)}
\newcommand{\Lts}{L^2(\Omega)^2}
\newcommand{\Ltc}{L^2(\Omega)^3}
\newcommand{\Ho}{H^1(\Omega)}
\newcommand{\Hoh}{H^1(\wh{\Omega})}
\newcommand{\Hoi}{H^1(\Omega_i)}
\newcommand{\Hos}{H^1(\Omega)^2}
\newcommand{\Hoc}{H^1(\Omega)^3}
\newcommand{\Hoch}{H^1(\wh{\Omega})^3}
\newcommand{\Hoci}{H^1(\Omega_i)^3}
\newcommand{\Hoz}{H^1_0(\Omega)}
\newcommand{\Ht}{H^2(\Omega)}
\newcommand{\Hti}{H^2(\Omega_i)}
\newcommand{\Hts}{H^2(\Omega)^2}
\newcommand{\Htc}{H^2(\Omega)^3}
\newcommand{\Htz}{H^0(\Omega)}
\newcommand{\Hh}{H^{1/2}(\Gamma)}
\newcommand{\Hhi}{H^{1/2}(\Gamma_i)}
\newcommand{\Hmh}{H^{-1/2}(\Gamma)}
\newcommand{\Hdiv}{H(\divvr;\,\Omega)}
\newcommand{\Hdivh}{H(\divv;\,\wh \Omega)}
\newcommand{\hcurl}{H(\curl\,A;\,\Omega)}
\newcommand{\Hcurl}{H(\curll\,A;\,\Omega)}
\newcommand{\Hcrl}{H(\curll\,;\,\Omega)}
\newcommand{\hcrl}{H(\curl\,;\,\Omega)}
\newcommand{\Hcrlh}{H(\curll\,;\,\wh\Omega)}
\newcommand{\hcrlh}{H(\curl\,;\,\wh\Omega)}
\newcommand{\Wdiv}{\BW_0(\mbox{\divv}\,;\,\Omega)}
\newcommand{\Wcurl}{\BW_0(\mbox{\curl}\,A;\,\Omega)}
\newcommand{\WcrossV}{\BW \times V}